\documentclass[12pt, reqno]{amsart}
\usepackage{amsmath,amsthm,amssymb,hyperref,mathtools}
\usepackage{enumerate}
\usepackage{tikz-cd}
\usepackage{amsaddr}
\usepackage{amscd}
\usepackage[all]{xy}
\usepackage{color}

\usepackage{comment}

\newtheorem{theorem}{Theorem}[section]
\newtheorem{lemma}[theorem]{Lemma}
\newtheorem{corollary}[theorem]{Corollary}
\newtheorem{proposition}[theorem]{Proposition}

\theoremstyle{definition}
\newtheorem{definition}[theorem]{Definition}

\newtheorem{remark}[theorem]{Remark}

\numberwithin{equation}{section}

\DeclareMathOperator{\supp}{supp}
\DeclareMathOperator{\End}{End}
\DeclareMathOperator{\Ind}{Ind}
\DeclareMathOperator{\id}{id}

\newcommand\N{\mathbb N}
\newcommand\C{\mathbb C}
\newcommand\R{\mathbb R}
\newcommand\Y{\mathbb Y}
\newcommand\Z{\mathbb Z}

\newcommand\al{\alpha}
\newcommand\be{\beta}
\newcommand\Ga{\Gamma}
\newcommand\ga{\gamma}
\newcommand\de{\delta}
\newcommand\De{\Delta}
\newcommand\la{\lambda}
\newcommand\lan{{\la[n]}}
\newcommand\Om{\Omega}
\newcommand\om{\omega}
\newcommand\eps{\varepsilon}
\newcommand\epsb{\bar\eps}
\renewcommand\th{\theta}
\newcommand\ka{\varkappa}

\newcommand\T{\mathcal T}
\newcommand\HH{{\mathcal H}}
\newcommand\ZZ{\mathcal Z}
\newcommand\A{{\mathcal A}}
\newcommand\B{{\mathcal B}}

\newcommand\cond{\operatorname{cond}}
\newcommand\Sym{\operatorname{Sym}}
\newcommand\gr{\operatorname{gr}}
\newcommand\tr{\operatorname{tr}}
\newcommand\Hom{\operatorname{Hom}}
\newcommand\SSym{{\Sym^*}}

\newcommand\F{F}
\newcommand\da{\downarrow}

\newcommand\bla{{\boldsymbol{\lambda}}}
\newcommand\bmu{{\boldsymbol{\mu}}}
\newcommand\bnu{{\boldsymbol{\nu}}}
\newcommand\beps{\bar\eps}

\newcommand\wt{\widetilde}
\newcommand\wh{\widehat}
\newcommand\p{{\bar p}}
\newcommand\q{{\bar q}}
\newcommand\Tb{\overline T}

\newcommand\zz{c}
\newcommand\rr{r}

\begin{document}

\title[]{Limits of group algebras for growing symmetric groups and wreath products}

\author{Irina Devyatkova$^{1,2}$ \and Grigori Olshanski$^{1,2}$}

\thanks{ 
\leftline{e-mail: irina.e.devyatkova@gmail.com, olsh2007@gmail.com}
\leftline{${}^1$ Igor Krichever Center for Advanced Studies, Skoltech, Moscow, Russia.} 
\leftline{${}^2$ Faculty of Mathematics, HSE University, Moscow, Russia.}
\leftline{Both authors supported by the Russian Science Foundation under project 23-11-00150}
}

\begin{abstract}
Let $S(\infty)$ denote the infinite symmetric group formed by the finitary permutations of the set of natural numbers; this is a countable group. We introduce its virtual group algebra, a completion of the conventional group algebra $\C[S(\infty)]$. The virtual group algebra is obtained by taking large-$n$ limits of the finite-dimensional group algebras $\C[S(n)]$ in the so-called tame representations of $S(\infty)$. (Note that our virtual group algebra is very different from the $C^*$-envelope.) We describe the structure of the virtual group algebra, which reveals a connection with Drinfeld-Lusztig degenerate affine Hecke algebras. Then we extend the results to wreath products $G\wr S(\infty)$ with arbitrary finite groups $G$.  
\end{abstract}

\date{}

\maketitle

\tableofcontents

\section{Introduction}

\subsection{Statement of problem}\label{sect1.1}

Given a finite or countable set $X$, let $S(X)$ denote the group of all permutations of $X$. We assume that $S(X)$ act on $X$ on the left. It is convenient to represent permutations of $X$ as \emph{strictly monomial} 0-1 matrices of the format $X\times X$. Namely, given $g\in S(X)$ and a couple of indices $(i,j)\in X\times X$,  the $(i,j)$-entry of the matrix of $g$ is
$$
g_{ij}=\begin{cases} 1, & \text{if  $gj=i$,}\\ 0, & \text{otherwise.} \end{cases}
$$

Throughout the paper we will exclusively deal with the sets $\N:=\{1,2,\dots\}$ and $\N_n:=\{1,\dots,n\}$, where $n=1,2,\dots$\,. We abbreviate $S(n):=S(\N_n)$ and identify $S(n)$ with the subgroup of $S(\N)$ fixing the points $n+1,n+2,\dots$\,. 

Next, we set
$$
S(\infty):=\varinjlim S(n)=\bigcup_{n=1}^\infty S(n) \subset S(\N).
$$
The elements $g\in S(\infty)$ are the \emph{finitary} permutations of $\N$ (that is, permutations moving only finitely many points). In other words, $g_{ij}\ne\de_{ij}$ for finitely many entries only. 
 
The group $S(\N)$ is endowed with a natural topology making it a  topological group. This  topology is totally disconnected; it is determined by the property that the pointwise stabilizers of finite subsets form a fundamental system of neighborhoods of the identity. Note that $S(\N)$ is a basic example of an \emph{oligomorphic group}, see \cite{Cam90}, \cite{Tsa12}. 

The unitary representations of the topological group $S(\N)$ were first studied by Lieberman \cite{Lie72}. He showed that that there are countably many irreducible representations (up to equivalence) and any unitary representation is a discrete direct sum of irreducible ones (for other approaches, see \cite{Ols85}, \cite{Tsa12}). In this respect, the group $S(\N)$ behaves like a compact group, although it is neither compact nor even locally compact. Moreover, there are many similarities between the representations of $S(\N)$ and those of the finite symmetric groups $S(n)$.  

The group $S(\infty)$ is dense in $S(\N)$, so any unitary representation of $S(\N)$ is uniquely determined by its restriction to $S(\infty)$. The representations of $S(\infty)$ obtained in this way are called \emph{tame}; they can also be characterized intrinsically,  without recourse to the group $S(\N)$. Thus, instead of representations of $S(\N)$ we can equally well deal with tame representations of $S(\infty)$. 

Because $S(\infty)$ is a countable group, its group algebra $\C[S(\infty)]$ is well defined, and any unitary representation $S(\infty)$, not necessarily a tame one, admits an extension to this algebra. However,  if we focus on tame representations of $S(\infty)$ and try to exploit their similarity with representations of the groups $S(n)$, we discover that the algebra $\C[S(\infty)]$ does not fit well into our picture: it seems to be too small. For instance, the group algebra $\C[S(n)]$ has a large center which separates irreducible representations, whereas the center of $\C[S(\infty)]$ is trivial.  This fact prompts us to look for a suitable completion of $\C[S(\infty)]$, which would possess a richer structure. The aim of the paper is to define such a completion and describe its structure. 

In fact, we do this in greater generality --- for wreath products $G \wr S(\infty)$, where $G$ is an arbitrary finite group. However, for the sake of simplicity, in this introductory section, we mainly focus on the case of the group $S(\infty)$.  

\subsection{The virtual group algebra of $S(\infty)$}\label{sect1.2}

Let  $\al=(\al_1,\al_2,\dots)$ be an infinite sequence of elements of $\C[S(\infty)]$ such that $\al_n\in\C[S(n)]$. We say that $\al$ is \emph{convergent} if:

\begin{itemize}
\item the elements $\al_n$ are uniformly bounded with respect to a certain (double) filtration in the spaces $\C[S(n)]$ (see Definition \ref{def6.A});
\item for any irreducible tame representation $T$, the operators $T(\al_n)$  have a strong limit, which is denoted by  $\Tb(\al)$.
\end{itemize}
Next, we say that two convergent sequences, $\al$ and $\be$, are \emph{equivalent} if $\Tb(\al)=\Tb(\be)$ for any irreducible tame $T$.

Consider now the  set of equivalence classes of convergent sequences and denote it by $\B$; it is an associative algebra, where  all operations are defined term-wise. Thus, $\B$ is a kind of \emph{large-$n$ limit} of the group algebras $\C[S(n)]$. 

There is a natural embedding $\C[S(\infty)]\to\B$: its image is formed by the \emph{stable} sequences $\al=(\al_n)$ (we say that $\al$ is stable if $\al_n$ does not depend on $n$ for all $n$ large enough). We call $\B$ the \emph{virtual group algebra of $S(\infty)$}. 

By the very definition, the algebra $\B$ acts in the space of each irreducible tame representation $T$ via the correspondence $\al\mapsto \Tb(\al)$, and this action extends the natural action of the group algebra $\C[S(\infty)]$. 

Note that  $\B$ is very different from $C^*(S(\infty))$, the $C^*$-envelope of the group algebra $\C[S(\infty)]$.

\subsection{The centralizer construction}\label{sect1.3}
In Theorem \ref{thm1.A} below, we establish a connection with the \emph{centralizer construction} of \cite{MO01}. That construction employs the \emph{symmetric inverse semigroups} $\Ga(n)\supset S(n)$ (aka \emph{rook monoids}) and produces an algebra $\A\supset \C[S(\infty)]$, which has the following structure: 

\begin{itemize}

\item the center of $\A$, denoted by $\A_0$, is isomorphic to a close relative of the algebra of symmetric functions $\Sym$, \emph{the algebra of shifted symmetric functions}, denoted by  $\Sym^*$ (see section \ref{sect3.1} for the definition);  

\item the whole algebra $\A$ splits into the tensor product of $\A_0$ and another subalgebra, denoted by  $\widetilde{\mathcal H}$:
\begin{equation}\label{eq1.A}
\A\simeq\A_0\otimes\wt\HH;
\end{equation}

\item the algebra $\wt\HH$ is the union of an ascending chain of subalgebras $\wt\HH_m$, $m=1,2,\dots$, where $\wt\HH_m$ is a relative of the \emph{degenerate affine Hecke algebra} $\HH_m$ associated with the symmetric group $S(m)$.
\end{itemize}

A precise description of $\wt\HH_m$ is given in section \ref{sect5.3}  below.  Here are a few comments to the above rough description. 

The degenerate affine Hecke algebras (aka graded affine Hecke algebras) were introduced by Drinfeld \cite{Dri86} and Lusztig \cite{Lus89}. (About connections between the works \cite{Dri86} and \cite{Lus89}, see \cite[section 4]{RS03}.)
The algebra $\HH_m$ contains $\C[S(m)]$ and the algebra of polynomials in $m$ variables $\C[u_1,\dots,u_m]$. As a vector space, $\HH_m$ is isomorphic to the tensor product $\C[S(m)]\otimes \C[u_1,\dots,u_m]$ but there are nontrivial commutation relations between permutations and polynomials, of the form 
$$
s u_i s^{-1} = u_{s(i)} + X(s,i), \qquad s\in S(m), \quad i=1,\dots,m,
$$
where $X(s,i)$ are certain elements of  $\C[S(m)]$. For instance, if $s$ is the elementary transposition $(i,i+1)$, then $X(s,i)=(i,i+1)$. Without the terms $X(s,i)$, we would obtain a simpler object, the crossed product algebra $S(m)\ltimes\C[u_1,\dots,u_m]$; the algebra $\HH_m$ is a nontrivial deformation of this crossed product. 
As for the algebra $\wt\HH_m$, it is a certain modification of $\HH_m$, in which the role of the group $S(m)$ is played by the semigroup $\Ga(m)$.  

\subsection{Main result (case of $S(\infty)$)}

\begin{theorem}\label{thm1.A}
The algebra $\B$ {\rm(}the virtual group algebra of $S(\infty)${\rm)} is isomorphic to the algebra $\A$ defined by the centralizer construction.
\end{theorem}

For a more detailed formulation, see Theorem \ref{thm6.A}. 

As a corollary we obtain an explicit presentation of $\B$ in terms of countably many generators and defining relations. It follows, in particular, that $\B$ has countable dimension, which is not obvious from the initial definition. 

The definition of the algebra $\A$ is constructive, which makes it possible to describe its structure, while the definition of the algebra $\B$ is non-constructive (it can be compared with the completion of a metric space).  On the other hand, from the viewpoint of representation theory, the algebra $\B$ looks as a reasonable version of group algebra. The meaning of Theorem \ref{thm1.A} is that the two definitions, very different in spirit, lead to the same object.  

\subsection{Generalization to wreath products}

All results described above hold in greater generality --- namely,  for the wreath product groups  $G\wr S(\infty)$, where $G$ is an arbitrary finite group. We prove this by making use of Jinkui Wan's work \cite{Wan10}, where the centralizer construction of \cite{MO01} is extended to the wreath products $G\wr S(\infty)$. Our final result, a counterpart of Theorem \ref{thm1.A}, is stated in Theorem \ref{thm11.A}.
 
\subsection{Analogy with infinite-dimensional classical groups}

Similar results were earlier obtained for the three series $O(n)$, $U(n)$, $Sp(n)$ of compact classical groups. Consider the corresponding  inductive limit groups
\begin{equation}\label{eq1.C}
O(\infty):=\varinjlim O(n), \quad U(\infty):=\varinjlim U(n), \quad Sp(\infty):=\varinjlim Sp(n).
\end{equation}
Let $\mathcal G$ denote any of these three infinite-dimensional groups. For $\mathcal G$, there is again a distinguished class of unitary representations called tame, and these are precisely those representations that can be extended to a certain topological completion $\overline{\mathcal G}\supset\mathcal G$.   

For definiteness, let us focus on the orthogonal case: $\mathcal G=O(\infty)$. Then the counterparts of the group $S(\N)$, the group algebra $\C[S(n)]$, and the group algebra $\C[S(\infty)]$ are, respectively, the following objects:
\begin{itemize}
\item 
the group $\overline{\mathcal G}=O(\ell^2(\N))$ of all orthogonal operators on the coordinate real Hilbert space $\ell^2(\N)$, and the relevant topology in $O(\ell^2(\N))$ is the strong (=weak) operator topology;
\item
the universal enveloping algebra $\mathcal U(\mathfrak o(n,\C))$, where $\mathfrak o(n,\C)$ is the Lie algebra of $n\times n$ complex skew-symmetric matrices;
\item
the associative algebra
$$
\mathcal U(\mathfrak o(\infty,\C))=\varinjlim\, \mathcal U(\mathfrak o(n,\C)).
$$
\end{itemize}

As shown in \cite{Ols91a}, there exists an associative algebra $A\supset \mathcal U(\mathfrak o(\infty,\C))$ which acts on all irreducible tame representations of $O(\infty)$ and has a large center $A_0$ separating these representations. (Note that the center of $\mathcal U(\mathfrak o(\infty,\C))$ is trivial, as in the case of $\C[S(\infty)]$.) The algebra $A$ can be defined in two different but equivalent ways: (1) as a large-$n$ limit of the algebras $\mathcal U(\mathfrak o(n,\C))$, where the limit should be understood in a suitable way, and (2) by making use of the centralizer construction. 

The first definition seems to be the most natural, but it is not constructive, while the the second one allows one to describe the structure of the limit algebra $A$. Namely, it turns out that  
\begin{equation}\label{eq1.B}
A\simeq A_0\otimes Y,
\end{equation}
where the center $A_0$ is again isomorphic to the algebra $\Sym^*$,  and
$$
Y:=\varinjlim Y_m,
$$
where $Y_m$ denotes  the \emph{Yangian} of the Lie algebra $\mathfrak{gl}(m,\C)$ (about the Yangians $Y_m$, see \cite{MNO96}, \cite{Mol07}).

An amazing fact is that the first approach deals with the orthogonal Lie algebras, while the second approach (the centralizer construction) requires the use of the general linear algebras. An explanation of this effect is given in \cite[section 1.4]{Ols91a}. 

Note that in the symmetric group context, the role of the general lineal Lie algebras is played by the symmetric inverse semigroups $\Ga(n)$.

The method of \cite{MO01} is a version of the second approach (the centralizer construction), while in the present paper we develop the first approach (the large-$n$ limit transition) and then establish a connection with the centralizer construction of \cite{MO01}.  

The evident similarity between \eqref{eq1.A} and \eqref{eq1.B} is a manifestation of a surprising parallelism between the representation theory of the group $S(\infty)$ and the representation theory of the infinite-dimensional classical groups \eqref{eq1.C}. In this connection see also \cite{BO13}.  

These two results, \eqref{eq1.A} and \eqref{eq1.B}, show that two kinds of `affine' objects of type A, the degenerate affine Hecke algebras and the Yangians, can be obtained from `non-affine' objects (symmetric groups $S(n)$ and universal enveloping algebras $\mathcal U(\mathfrak o(n,\C))$) by a large-$n$ limit transition. 

Note also that the degenerate affine Hecke algebras  and the Yangians  meet in Drinfeld's duality described in \cite{Dri86}.

The main results, both for the classical groups and for the symmetric groups, were initially announced in the short note \cite{Ols88}.

\subsection{Organization of the paper} 

In the first part of the paper (sections \ref{sect2}--\ref{sect6}) we deal with the group $S(\infty)$, and in the second part (sections \ref{sect7}--\ref{sect11}) we extend the results to the wreath products $G(\infty):=G\wr S(\infty)$. 

The general plan of the proof in the second part is similar to that of the first part. However, the case of wreath products $G(\infty)$ is technically more sophisticated; this is why we decided to examine the group $S(\infty)$ first. 

The core of the second part is the computation in Proposition \ref{prop8.A}. The results of Jinkui Wan's paper \cite{Wan10} are used in section \ref{sect10}.

\section{Tame representations of $S(\infty)$}\label{sect2}

We introduce tame representations and describe the irreducible tame representations of $S(\infty)$ in two different ways.

\subsection{Definition of tame representations}\label{sect2.1}

Recall that we regard the elements of the symmetric group $S(n)$ as $n\times n$ permutation matrices and also consider $S(n)$ as a subgroup of $S(\N)$. This allows us to form the inductive limit group $S(\infty)=\varinjlim S(n)$. 

Given two integers $m$ and $n$ such that $0\le m< n$,  we denote by $S_m(n)$ the subgroup in $S(n)$ formed by the $n\times n$ permutation matrices with the first $m$ diagonal entries equal to $1$. Note that $S_m(n)$ is isomorphic to $S(n-m)$. Next, we set
$$
S_m(\infty) = \bigcup_{n> m} S_m(n), \quad m=0, 1,2,\dots\,.
$$
The subgroups $S_m(\infty)$ form a decreasing chain in $S(\infty)$, 
$$
S(\infty)=S_0(\infty)\supset S_1(\infty)\supset S_2(\infty) \supset \ldots,
$$
with 
$$
\bigcap_m S_m(\infty) = \{e\}.
$$

Given a unitary representation $T$ of $S(\infty)$, we denote by $H(T)$ the Hilbert space of $T$, and for $m\ge 0$ we denote by $H_m(T)$ the  subspace of $S_m(\infty)$-invariant vectors in $H(T)$. These spaces form an ascending chain and we set
$$
H_{\infty}(T) := \bigcup_{m=0}^\infty H_m(T).
$$
Note that $H_m(T)$ is an $S(m)$-invariant subspace, so $H_{\infty}(T)$ is invariant with respect to the action of the whole group $S(\infty)$.

\begin{definition}\label{def2.A}
(i) A unitary representation $T$ of the group $S(\infty)$ is said to be \emph{tame} if  $H_{\infty}(T)$ is dense in $H(T)$.

(ii) If $T$ is tame, then the minimal number $m$ such that $H_m(T)\ne\{0\}$ is called the \emph{conductor} of $T$ and is denoted by $\cond(T)$.
\end{definition}

\subsection{The irreducible representations $T^\la$}\label{sect2.2}
We denote by $\mathbb{Y}_k$ the set of Young diagrams with $k$ boxes and write 
$$
\mathbb{Y} = \bigsqcup_{k\ge0} \mathbb{Y}_k
$$
for the set of all Young diagrams (the set $\mathbb{Y}_0$ consists of a single empty Young diagram $\varnothing$). Given $\la\in\Y$, we denote by $|\la|$ the number of boxes in $\la$. 
For each $\la\in\Y$ with $|\la|=\ell>0$ we denote by $\pi^\la$ the irreducible representation of the symmetric group $S(\ell)$ indexed by $\la$. 

We are going to assign to each $\la\in\Y$ a unitary representation $T^\la$ of the group $S(\infty)$. If $\la=\varnothing$, then $T^\varnothing$ is the trivial one-dimensional representation. Assume now that $\la\ne\varnothing$ and set $\ell:=|\la|$. Then we define $T^\la$ as an induced representation:
\begin{equation}\label{eq2.E}
T^\la:=\Ind^{S(\infty)}_{S(\ell)S_\ell(\infty)} (\pi^\la\otimes 1).
\end{equation}

Let us comment on the notation. Here the group $S(\ell)S_\ell(\infty)\subset S(\infty)$ is the product of two commuting subgroups, $S(\ell)$ and $S_\ell(\infty)$; this group is formed by the (finitary) permutations preserving the subset $\{1,\dots,\ell\}\subset\N$ as a whole. By $\pi^\la\otimes 1$ we denote the (outer) product of the irreducible representation $\pi^\la$ of $S(\ell)$ and the trivial representation of $S_\ell(\infty)$. The representation $T^\la$ is induced from the subgroup $S(\ell)S_\ell(\infty)$ to the group $S(\infty)$.

Here is a more detailed description of  the induced representation \eqref{eq2.E}. Let $\Om(\ell,\infty)$ denote the set of matrices of the format $\N_\ell\times\N$, with the entries in $\{0,1\}$, and such that each row contains precisely one $1$ and no column contains two or more $1$'s. The group $S(\ell)$ acts on $\Om(\ell,\infty)$ on the left by permutations of rows, while the group $S(\infty)$ acts on the right by permutations of columns. 

Next, let $H(\pi^\la)$ denote the  space of $\pi^\la$ and consider the Hilbert space 
\begin{equation}\label{eq2.M}
\ell^2(\Om(\ell,\infty); H(\pi^\la))=\ell^2(\Om(\ell,\infty))\otimes H(\pi^\la);
\end{equation}
its elements are $H(\pi^\la)$-valued vector functions on $\Om(\ell,\infty)$ with the scalar product
\begin{equation}\label{eq2.K}
(\phi_1,\phi_2):=\sum_{\om\in\Om(\ell,\infty)} (\phi_1(\om),\phi_2(\om)),
\end{equation}
where $(-,-)$ on the right-hand side is the $S(\ell)$-invariant scalar product in $H(\pi^\la)$. The Hilbert space $H(T^\la)$ of $T^\la$ is defined as the subspace of the space \eqref{eq2.M} formed by the functions $\phi(\om)$ satisfying the additional condition
\begin{equation}\label{eq2.L}
\phi(s\om)=\pi^\la(s)\phi(\om), \quad s\in S(\ell), \; \om\in\Om(\ell,\infty).
\end{equation}
The action of $S(\infty)$ on this subspace is given by
\begin{equation}\label{eq2.A}
(T^\la(g) \phi)(\om)=\phi(\om g), \quad g\in S(\infty), \; \om\in\Om(\ell,\infty).
\end{equation}
Using the obvious identification of $\Om(\ell,\infty)$ with the homogeneous space $S_\ell(\infty)\setminus S(\infty)$ one sees that \eqref{eq2.A} is equivalent to the initial definition \eqref{eq2.E}.

The subspaces $H_n(T^\la)$ of $S_n(\infty)$-invariants are described as follows (see \cite{Ols85} , \cite{MO01}):
\begin{itemize}
\item if $n<\ell$, then $H_n(T^\la)=\{0\}$;

\item if $n\ge \ell$, then $H_n(T^\la)$ is formed by the functions $\phi(\om)$ concentrated on the subset of matrices $\om\in \Om(\ell,\infty)$ whose columns with indices greater than $n$ are filled entirely of zeroes. 
\end{itemize}
From this description it follows that $T^\la$ is tame and
\begin{equation}\label{eq2.J}
\cond(T^\la)=|\la|, \quad \la\in\Y.
\end{equation}

The next theorem shows that the tame representations behave as representations of compact groups. 

\begin{theorem}\label{thm2.A}

{\rm(i)} The representations  $T^\la$, $\la\in\Y$, are irreducible, pairwise non-equivalent, and they exhaust all irreducible tame representations of $S(\infty)$.  

{\rm(ii)} Any tame representation of $S(\infty)$ can be decomposed into a discrete direct sum of irreducible representations. 
\end{theorem}

\begin{proof}
See \cite{Ols85} and \cite[section 2]{MO01}.
\end{proof}

The above construction works, without changes, when the discrete group $S(\infty)$ is replaced by the topological group $S(\N)$. As mentioned in the introduction, tame representations of $S(\infty)$ and continuous unitary representations of $S(\N)$ are the same (for the proof, see \cite{Ols85}). In the context of the topological group $S(\N)$, the unitary representations were first studied by Lieberman \cite{Lie72}; the method of \cite{Ols85} is different; yet another approach is developed by Tsankov \cite{Tsa12}.

\subsection{The spectrum of $T^\la\big|_{S(n)}$}\label{sect2.3}

We identify Young diagrams $\la\in\Y$ with the corresponding partitions $(\la_1,\la_2,\dots)$. 
Fix an arbitrary $\la\in\Y$. For $n\ge|\la|$, let $\lan\in\Y_n$ denote the diagram obtained from $\la$ by adding a new row of length $n-|\la|$. If $n\ge |\la|+\la_1$, then 
\begin{equation}\label{eq2.F}
\lan=(n - |\la|, \lambda_1, \lambda_2, \ldots),
\end{equation}
otherwise the new coordinate $n-|\la|$ appears somewhere in between the coordinates of $\la$.  

Introduce a notation: for two partitions $\la$ and $\nu$, we write $\nu\succ\la$ or $\la\prec\nu$ if
$$
\nu_1\ge \la_1\ge \nu_2\ge\la_2\ge\dots\,.
$$

Equivalently, $\nu\succ\la$ means that $|\nu|\ge|\la|$ and the diagram $\nu$ can be obtained from the diagram $\la$ by appending a horizontal strip of length $|\nu|-|\la|$. 

A simple but important remark is  that $\nu\succ\la$ implies $\la\succ (\nu_2,\nu_3,\dots)$. We use it in what follows. 

Let $\ell:=|\la|$. For $n>\ell$, we set
\begin{equation}\label{eq2.H}
X_n(\la):=\{\nu\in\Y_n: \nu\succ\la\}.
\end{equation}

We assume now that $n\ge |\la|+\la_1$, so that \eqref{eq2.F} holds. Then $X_n(\la)$ contains the diagram $\la[n]$. Note that among all diagrams $\nu\in X_n(\la)$, the diagram $\la[n]$ is the one with the shortest possible first row and hence the largest possible sum of the other row lengths. More precisely, because  $\nu\in X_n(\la)$ entails $\la\succ (\nu_2,\nu_3,\dots)$, all diagrams $\nu\in X_n(\la)$, distinct from $\la[n]$, are precisely of the form $\mu[n]$, where $\mu\prec \la$, $\mu\ne\la$. In particular, the cardinality of the set $X_n(\la)$ does not depend on $n$. 

In the proof of the next proposition we use the classic  \emph{Young branching rule}. It says that if $\ka\in\Y_k$, where $k\ge2$, then 
\begin{equation}\label{eq2.D}
\pi^\ka\big|_{S(k-1)}\; \sim\; \bigoplus_{\nu:\, \nu\nearrow\ka} \pi^\nu,
\end{equation}
where the relation $\nu\nearrow\ka$ means that the diagram $\nu$ can be obtained from the diagram $\ka$ by removing a box. In other words, $\nu\subset\ka$ and $|\nu|=|\ka|-1$

\begin{proposition}\label{prop2.A}
Let $\la\in\Y_\ell$ and $n\ge \ell+\la_1$. Then the restriction of $T^\la$ to the subgroup $S(n)$ contains $\pi^\lan$ with multiplicity $1$, while all other irreducible components of $T^\la\big|_{S(n)}$ are of the form $\pi^{\mu[n]}$, where $\mu\subset\la$, meaning that the diagram $\mu$ is strictly contained in the diagram $\la$. 
\end{proposition}

\begin{proof}
Examine first the representation of the subgroup $S(n)\subset S(\infty)$ on the subspace $H_n(T^\la)$. From the definition of $T^\la$ it follows that this representation is equivalent to the induced representation $\Ind^{S(n)}_{S(\ell)S_\ell(n)}(\pi^\la\otimes 1)$. The structure of this representation is well known:
\begin{equation}\label{eq2.I}
\Ind^{S(n)}_{S(\ell)S_\ell(n)}(\pi^\la\otimes 1)\sim\bigoplus_{\nu\in X_n(\la)} \pi^\nu.
\end{equation}
In particular, this induced representation contains $\pi^\lan$. 

Suppose now that $\nu\in\Y_n$ is such that $\pi^\nu$ is contained in $T^\la\big|_{S(n)}$. We claim that then $\pi^\nu$ enters the decomposition of the space $H_k(T^\la)$ under the action of $S(k)$, for all $k$ large enough. 

Indeed, the assumption on $\nu$ means that there exists a nonzero intertwining operator $A\in\Hom_{S(n)}(H(\pi^\nu), H(T^\la))$. Let $P_k$ denote the orthogonal projection operator from $H(T^\la)$ to $H_k(T^\la)$. As $k\to\infty$, the operators $P_k$ strongly converge to the identity operator. It follows that $P_kA\ne0$ provided $k$ is large enough, which proves our claim. 

Now we fix such a large $k$. By virtue of \eqref{eq2.I} (where we replace $n$ with $k$), all irreducible components of the space $H_k(T^\la)$ under the action of $S(k)$ have the form $\pi^\ka$ with $\ka\in X_k(\la)$. This implies the inequalities
$$
\ka_{i+1}\le \la_i, \quad i=1,2,\dots\,.
$$

On the other hand, there exists $\ka$ such that $\pi^\nu$ enters the decomposition of $\pi^\ka\big|_{S(n)}$. By the Young branching rule, this entails $\nu\subseteq\ka$, so that 
$$
\nu_{i+1}\le \la_i, \quad i=1,2,\dots\,.
$$
But this precisely means that $\nu$ has the form $\mu[n]$ for some $\mu\subseteq\la$. 

This completes the proof. 
\end{proof}

\begin{remark}\label{rem2.A}
As $n$ gets large, the irreducible component $\pi^\lan$ takes almost all the space of $\Ind^{S(n)}_{S(\ell)S_\ell(n)}(\pi^\la\otimes 1)$ in the sense that 
\begin{equation}\label{eq2.G}
\lim_{n\to\infty} \dfrac{\dim \pi^\lan}{\dim\Ind^{S(n)}_{S(\ell)S_\ell(n)}(\pi^\la\otimes 1)}=1
\end{equation}
(this relation can be easily proved by using, say, the hook formula). 
\end{remark}

\subsection{Another realization of representations $T^\la$}\label{sect2.4}

Fix $\la\in\Y$ and let $n$ be large enough. By virtue of the Young branching rule, $\pi^{\lan}$ is contained with multiplicity one in the restriction of $\pi^{\la[n+1]}$ to the subgroup $S(n)\subset S(n+1)$. This enables us to define, in a canonical way, the inductive limit representation 
$$
\Pi^\la:=\varinjlim \pi^\lan
$$
of the group $S(\infty)$. 

\begin{proposition}[\cite{Ols80}, Theorem 2.7;  \cite{MO01}, Proposition 2.5] \label{prop2.B}
For each $\la\in\Y$, the representations $T^\la$ and $\Pi^\la$ are equivalent.
\end{proposition}

\begin{proof}
From Proposition \ref{prop2.A} we know that $\pi^\lan$ enters $T^\la\big|_{S(n)}$ with multiplicity $1$; let $H_n\subset H(T^\la)$ denote the corresponding subspace. We claim that $H_n$ is contained in $H_{n+1}$. Indeed, using the Young branching rule we see that $\pi^\lan$ is contained in the restriction of $\pi^{\la[n+1]}$ to $S(n)$ and not contained in the restriction of $\pi^{\mu[n+1]}$ if $\mu\subset\la$ and $\mu\ne\la$. By virtue of Proposition \ref{prop2.A} this means that $H_n\subset H_{n+1}$.

We have proved that $T^\la$ contains a subrepresentation equivalent to $\Pi^\la$. Because $T^\la$ is irreducible, we conclude that $T^\la\sim\Pi^\la$. 
\end{proof}

\subsection{Analog of Young's  branching rule for $T^\la$}\label{sect2.5}

The shift $i\mapsto i+1$ inside $\N$ gives rise to an endomorphism of the group $S(\infty)$, which we denote by $\xi$; it establishes an isomorphism between $S(\infty)$ and its subgroup $S_1(\infty)$:
\begin{equation}\label{eq2.C}
\xi: S(\infty)\to S_1(\infty)\subset S(\infty).
\end{equation}
We denote by $\xi^{-1}$ the inverse isomorphism $S_1(\infty)\to S(\infty)$. Its superposition with $T^\la$ produces a representation of $S_1(\infty)$ which we denote by $T^\la_1$. We also write $T_1^\la=T^\la\circ\xi^{-1}$.

\begin{proposition}[Branching rule]\label{prop2.C}
For $\la\in\Y$, the restriction of the representation $T^\la$ to the subgroup $S_1(\infty)$ is equivalent to a finite, multiplicity free direct sum of irreducible tame representations. Specifically, 
\begin{equation}\label{eq2.B}
T^\la\big|_{S_1(\infty)}\; \sim\; T^\la_1\oplus \bigoplus_{\mu:\, \mu\nearrow\la} T^\mu_1.
\end{equation} 
\end{proposition}

Because $T^\la$ is an induced representation, its restriction to the subgroup $S_1(\infty)$ can be decomposed by making use of the standard recipe. The reasoning below essentially follows this recipe.  

\begin{proof}
The case $\la=\varnothing$  is trivial (and then the sum over $\mu$ disappear), so we assume $|\la|=\ell\ge1$. We will work with the realization of section \ref{sect2.2}. 

Let us split the set $\Om(\ell,\infty)$ into two parts,
$$
\Om(\ell,\infty)=\Om(\ell,\infty)^0\sqcup \Om(\ell,\infty)',
$$
where $\Om(\ell,\infty)^0$ is formed by those matrices $\om\in\Om(\ell,\infty)$ that do not contain a `$1$' in the first column, and $\Om(\ell,\infty)'$ is the rest. According to this we obtain a decomposition of the Hilbert space $H(T^\la)$ into an orthogonal direct sum of two subspaces, say, $H^0$ and $H'$. 
This decomposition is invariant under the action of the subgroup $S_1(\infty)$. The subspace $H^0$ carries the representation $T^\la_1$. 

To decompose the subspace $H'$ we use the Young branching rule \eqref{eq2.D}; note that as $S(\ell-1)$ we can take the subgroup of $S(\ell)$ fixing an arbitrary given point $i\in\N_\ell$, not necessarily the last one.  

Observe now that each matrix $\om\in\Om(\ell,\infty)'$ contains a `$1$' in its first column, so the stabilizer of this column under the action of $S(\ell)$ is a subgroup of $S(\ell)$ isomorphic to $S(\ell-1)$; let us denote it by $S(\ell-1;\om)$. Now we can describe the decomposition of $H'$: its component corresponding to a given diagram $\mu\nearrow\la$ is formed by the functions $f\in H'$ satisfying the following condition: for each matrix $\om\in\Om(\ell,\infty)'$,  the vector $f(\om)\in H(\pi^\la)$ is transformed, under the action of the subgroup $S(\ell-1;\om)\subset S(\ell)$, according to the representation $\pi^\mu$.   
\end{proof}

More generally, for each $k=1,2,3,\dots$, we denote by $\xi^k$ the $k$th iteration of the endomorphism $\xi$ and regard it as an isomorphism of $S(\infty)$ onto its subgroup $S_k(\infty)$. Then we set
$$
T^\la_k:=T^\la\circ \xi^{-k}, \quad \la\in\Y.
$$
It is a representation of the group $S_k(\infty)$ acting on the same Hilbert space: $H(T^\la_k)=H(T^\la)$. On the other hand, denoting by $H_n(T_k^\la)$ the subspace of vectors, which are invariant under the subgroup $\xi^k(S_n(\infty))=S_{n+k}(\infty)$, we obviously have
$$
H_n(T^\la_k)=H_{n+k}(T^\la), \quad n=0,1,2,3,\dots\,.
$$
This implies in turn that, setting $H_\infty(T^\la_k):=\cup_{n\ge0}H_n(T^\la_k)$, we have
$$
H_\infty(T^\la_k)=H_\infty(T^\la), \quad n=0,1,2,3,\dots\,.
$$

From this observation and the branching rule \eqref{eq2.B} we obtain:

\begin{corollary}\label{cor2.A}
Let $\la\in\Y$ and $k=1,2,3,\dots$ be arbitrary. 

{\rm(i)} The restriction of $T^\la$ to the subgroup $S_m(\infty)$ is equivalent to a finite direct sum of irreducible representations of the form $T_k^\mu$, $\mu\subseteq\la$, taken with some multiplicities $d(\la:\mu)$. 

{\rm(ii)} Write the corresponding orthogonal decomposition of the Hilbert space $H(T^\la)$ in the form
\begin{equation}\label{eq2.N}
H(T^\la)=\bigoplus_{\mu:\, \mu\subseteq\la} V^\mu\otimes H(T_k^\mu),
\end{equation}
where $V^\mu$ are auxiliary finite-dimensional Hilbert spaces such that\, $\dim V^\mu=d(\la:\mu)$, with trivial action of the subgroup $S_k(\infty)$. Then 
$$
H_{n+k}(T^\la)=\bigoplus_{\mu:\, \mu\subseteq\la} V^\mu\otimes H_n(T_k^\mu), \quad n=0,1,2,\dots,
$$
and hence
$$
H_\infty(T^\la)=\bigoplus_{\mu:\, \mu\subseteq\la} V^\mu\otimes H_\infty(T_k^\mu).
$$
\end{corollary}

We use this corollary in the proof of Proposition \ref{prop5.A}.

\section{Virtual center of the group algebra $\C[S(\infty)]$}\label{sect3}

\subsection{Shifted symmetric functions}\label{sect3.1}

For more details about the material of this subsection see \cite{OO98} and \cite{BO17}. 

Let $\Sym(N)$ denote the algebra of symmetric polynomials in $N$ variables and $\Sym$ denote the algebra of symmetric functions (we are working over the ground field $\R$).  Recall that $\Sym$ can be defined as a projective limit,
\begin{equation*}
\Sym=\varprojlim \Sym(N),
\end{equation*}
taken in the category of graded algebras. Here the projection $\Sym(N)\to \Sym(N-1)$ is the specialization $x_N=0$. Equivalently, $\Sym$ can be defined as the algebra of formal power series in $x_1,x_2,\dots$, which have bounded degree and are symmetric with respect to finitary permutations of the variables. 

Now let us modify this definition. A polynomial in $N$ variables, say $x_1,\dots,x_N$, is said to be \emph{shifted symmetric} if it becomes symmetric in the new variables 
$y_i:=x_i+\sigma-i$, where $i=1,\dots,N$ and $\sigma$ is an arbitrary constant whose choice does not affect the definition. The shifted symmetric polynomials in $N$ variables form an algebra which we denote by $\SSym(N)$. It is a filtered algebra with respect to the degree of polynomials. Note that the associated graded algebra $\gr\SSym(N)$ is canonically isomorphic to $\Sym(N)$.  

\begin{definition} The \emph{algebra of shifted symmetric functions} is the projective limit
\begin{equation}\label{eq3.B}
\SSym=\varprojlim \SSym(N)
\end{equation}
taken in the category of filtered algebras, where the projection $\SSym(N)\to \SSym(N-1)$ is again the specialization $x_N=0$.
\end{definition}

Set $\R^\infty_0:=\varinjlim \R^N$; this is the set of infinite sequences $(x_1,x_2,\dots)$ of real numbers with finitely many nonzero terms. One can think of the elements  $f^*\in\SSym$  as the functions on $\R^\infty_0$ such that 
$$
f^*(x_1,\dots,x_N):=f^*(x_1,\dots,x_N,0,0,\dots) 
$$
is a shifted symmetric polynomial whose degree remains bounded as $N\to\infty$. 

Because $\Y\subset\R^\infty_0$, this allows us to interpret elements $f^*\in\SSym$ as functions on $\Y$: we set
$$
f^*(\nu)=f^*(\nu_1,\nu_2,\dots), \quad \nu\in\Y.
$$

For $k=0,1,2,\dots$, we denote by $\SSym^k\subset\SSym$ the $k$th term of the filtration of $\SSym$. We have a canonical linear isomorphism
$$
\SSym^k/\SSym^{k-1} \to \Sym^{[k]}, \quad k=1,2,\dots,
$$
where $\Sym^{[k]}$ stands for the $k$-th term of the grading of $\Sym$. If $f^*\in\SSym$ has degree $k$, we denote by $[f^*]$ its image in $\Sym^{[k]}$ and call it the \emph{highest term} of $f$.

Suppose $f_1,f_2,\dots$ is any system of algebraically independent homogeneous generators of $\Sym$, such that $\deg f_k=k$. Taking arbitrary elements $f^*_k\in \SSym^k$ with $[f^*_k]=f_k$ we obtain a system of generators of the algebra $\SSym$, which are consistent with its filtration. This means that any element of $\SSym^m$ is uniquely represented as a linear combination of the monomials of the form
$$
\prod_{k=1}^\infty (f^*_k)^{m_k}, \quad \sum_k k m_k\le m,
$$
with the understanding that the monomial with $m_1=m_2=\dots=0$ is equal to $1$. 

For instance, take as $f_k$ the power sum symmetric functions $p_k$. A possible way to lift them to $\SSym$ is to take as $f^*_k$ the shifted symmetric functions 
$$
p^*_{k,\sigma} (x_1,x_2,\dots):=\sum_{i=1}^\infty\left((x_i+\sigma-i)^k-(\sigma-i)^k\right), \quad k=1,2,\dots,
$$
where $\sigma$ is an arbitrary fixed constant. 

A distinguished basis in $\SSym$ is formed by the \emph{shifted Schur functions} $s^*_\la$, $\la\in\Y$, see \cite{OO98}. Note that $[s^*_\la]=s_\la$, where $\{s_\la: \la\in\Y\}\subset\Sym$ is the basis of Schur functions. 

One more basis in $\SSym$ is formed by the shifted symmetric functions  $p^\#_\rho$, $\rho\in\Y\}$, see \cite[(15.9)]{OO98}. They are defined by
\begin{equation}\label{eq3.J}
p^\#_\rho:=\sum_{\la\in\Y_k} \chi^\la_\rho s^*_\la, \qquad \rho\in\Y_k.
\end{equation}
where $\{\chi^\la_\rho: \la,\rho\in\Y_k\}$ is the character table of the symmetric group $S(k)$. Note that $[p^\#_\rho]=p_\rho$, as is seen  from the Frobenius formula 
$$
p_\rho:=\sum_{\la\in\Y_k} \chi^\la_\rho s_\la, \qquad \rho\in\Y_k. 
$$

Set 
$$
n^{\da k}:=n(n-1)\dots(n-k+1)=\frac{n!}{(n-k)!}.
$$
For a Young diagram $\nu$, we denote by $\dim\nu$ the number of standard tableaux of the shape $\nu$; this is quantity is equal to $\dim\pi^\nu$. 

\begin{proposition}\label{prop3.B}
Let $\rho\in\Y_k$, $k\ge1$. For any $\nu\in\Y_n$ one has
\begin{equation}\label{eq3.K}
p^\#_\rho(\nu)=\begin{cases} \dfrac{n^{\da k}\chi^\nu_{\rho\cup 1^{n-k}}}{\dim\nu}, & n\ge k,\\
0, & n<k. \end{cases}
\end{equation}
\end{proposition}

\begin{proof}
The shifted Schur functions possess the so-called \emph{extra vanishing property}: $s^*_\la(\nu)=0$ unless $\la$ is contained in $\nu$; in particular,  $s^*_\la(\nu)=0$ if $|\nu|<|\la|$ (\cite[Theorem 3.1]{OO98}). Together with \eqref{eq3.J} this implies that 
$$
\text{$p^\#_\rho(\nu)=0$ if $n<k$}.
$$ 
Next, in the case $n\ge k$, the formula \eqref{eq3.K} turns into \cite[(15.21)]{OO98}.
\end{proof} 

In what follows we will need a particular case of the shifted symmetric functions $p^\#_\rho$; these are the elements $p^\#_k$ indexed by one-part partitions $(k)$, $k=1,2,\dots$\,. Note that $[p^\#_k]=p_k$. 

\subsection{Central elements $\zz^{(k)}_{n}\in\C[S(n)]$ and their eigenvalues}\label{sect3.2}

We denote by $\ZZ\C[S(n)]$ the center of the group algebra $\C[S(n)]$. Each central element $c\in\ZZ\C[S(n)]$ acts as a scalar operator in any irreducible representation $\pi^\la$, $\la\in\Y_n$. We denote by $\wh c(\la)$ the corresponding scalar, so that $\pi^\la(c)=\wh c(\la)\cdot 1$. 

Let $1\le k\le n$. Given an ordered $k$-tuple of pairwise distinct indices $i_1,\dots,i_k$  contained in $\N_n$, let $(i_1,\dots,i_k)\in S(n)$ denote the the cyclic permutation $i_1\to i_2\to\dots\to i_k\to i_1$ leaving fixed all other points of $\N_n$. We set
\begin{equation}\label{eq3.F1}
\zz^{(k)}_n:=\sum (i_1,\dots,i_k),
\end{equation}
summed over all such $k$-tuples. 

Note that 
\begin{equation}\label{eq3.F3}
\zz^{(1)}_n:=n\cdot1, 
\end{equation}
where $1$ stands for the identity element of $S(n)$.  

We also set
\begin{equation}\label{eq3.F2}
\zz^{(k)}_n:=0, \quad n<k.
\end{equation}

Obviously, $\zz^{(k)}_n\in\ZZ\C[S(n)]$, so the quantities $\wh\zz^{\,(k)}_n(\la)$, where $\la\in\Y_n$, are defined. 

\begin{proposition}\label{prop3.A}
One has
\begin{equation}\label{eq3.A1}
\wh\zz^{\,(k)}_n(\la)=p^\#_k(\la), \qquad k=1,2,\dots, \quad \la\in\Y_n.
\end{equation}
\end{proposition}

\begin{proof} 
If $k>n$, then $\wh\zz^{\,(k)}_n(\la)=p^\#_k(\la)=0$, so we assume $1\le k\le n$. 
Observe that the number of $k$-tuples of indices in \eqref{eq3.F1} equals $n^{\da k}$. We have 
$$
\wh\zz^{\,(k)}_n(\la)=\frac{\tr\pi^\la(\zz^{\,(k)}_n)}{\dim\la}=\frac{n^{\da k}\chi^\la_{k, 1^{n-k}}}{\dim\la}=p^\#_k(\nu),
$$
where the last equality is a particular case of \eqref{eq3.K}. 
\end{proof}

\subsection{Filtration in $\C[S(\infty)]$}\label{sect3.3}

For $g\in S(\infty)$, we denote its \emph{degree}, $\deg g$, as the number of diagonal entries $g_{ii}$ not equal to $1$; this is the same as the number of points in $\N$ that are not fixed by $g$. It is readily seen that 
\begin{equation}\label{eq3.H}
\deg (gh)\le \deg g+\deg h, \quad g,h\in S(\infty).
\end{equation}
We extend this definition to the group algebra $\C[S(\infty)]$ by setting 
\begin{equation}\label{eq3.H1}
\deg \left(\sum \al(g) g\right):=\max\{\deg g: \al(g)\ne0\}.
\end{equation}
From \eqref{eq3.H} it follows that the definition \eqref{eq3.H1} makes $\C[S(\infty)]$ a filtered algebra. Because $S(n)$ is a subgroup of $S(\infty)$, the group algebra $\C[S(n)]$ is a subalgebra of $\C[S(\infty)]$, so we obtain a filtration of $\C[S(n)]$, for each $n$.

In particular, the center $\ZZ\C[S(n)]$ also acquires a filtration. 

\begin{corollary}\label{cor3.A}
For any $f^*\in\SSym$, there exists a sequence of elements $\zz_n(f^*)\in\ZZ\C[S(n)]$, $n=1,2,\dots$, such that for any $n$ one has
\begin{equation}\label{eq3.H2}
\deg \zz_n(f^*)\le \deg f^*
\end{equation}
and 
\begin{equation}\label{eq3.A}
\wh{\zz_n(f^*)}(\la)=f^*(\la_1,\la_2,\dots), \quad \la\in\Y_n.
\end{equation}
\end{corollary}

\begin{proof}
This is a direct consequence of Proposition \ref{prop3.A} and the following facts: first, 
\begin{equation}\label{eq3.L}
\text{$\deg\zz^{(k)}_n=k$ for all $k\ge2$, while $\deg \zz^{(1)}_n=0$;}
\end{equation}
second, $[p^\#_k]=p_k$ and $\deg p_k=k$; third, the elements $p_k$ are homogeneous generators of the graded algebra $\Sym$. 
\end{proof}

\subsection{Construction of virtual central elements}\label{sect3.4}

\begin{definition}\label{def3.A}

(i) Let  $\al=(\al_n: n\ge n_0)$ be an infinite sequence of elements of $\C[S(\infty)]$ such that $\al_n\in\ZZ\C[S(n)]$ for all $n\ge n_0$. We say that $\al$ is \emph{convergent} if the following two conditions are satisfied:

\begin{itemize}
\item for any irreducible tame representation $T^\la$, $\la\in\Y$, the operators $T^\la(\al_n)$  have a limit,  $\Tb^\la(\al)$, in the strong operator topology; 
\item $\sup \deg \al_n<\infty$.
\end{itemize}

(ii) Next, we say that two convergent sequences, $\al$ and $\be$, are \emph{equivalent} if $\Tb^\la(\al)=\Tb^\la(\be)$ for any $\la\in\Y$. In particular, $\al$ and $\be$ are equivalent if $\al_n=\be_n$ for all $n$ large enough, so that the value of the initial index $n_0$ is irrelevant.  
\end{definition}

\begin{lemma}\label{lemma3.B}
The set of all convergent sequences {\rm(}with a fixed initial index $n_0${\rm)} is an algebra with respect to component-wise operations. 
\end{lemma}

\begin{proof}
This obviously holds true for the operations of multiplication by scalars and com\-po\-nent-wise addition. To show that the same holds for component-wise multiplication we observe that if $\al=(\al_n)$ is convergent, then, for any fixed $\la$, the norms of the operators $T^\la(\al_n)$ are uniformly bounded by  the Banach--Steinhaus theorem; then we use the fact that the operation of multiplication is jointly continuous in the strong topology, on bounded subsets of the space of operators on a Hilbert space.  
\end{proof}

\begin{definition}\label{def3.B}
Denote by $\ZZ$ the set of equivalence classes of convergent sequences of central elements $\al=(\al_n)$. We call $\ZZ$ the \emph{virtual center of\/ $\C[S(\infty)]$}. By virtue of Lemma \ref{lemma3.B}, the component-wise operations on convergent sequences make $\ZZ$ an algebra. 
\end{definition}

If $\al$ is convergent, then $T^\la(\al_n)$ commutes with any fixed element of $\C[S(\infty)]$ provided $n$ is large enough. Since $T^\la$ is irreducible, $\Tb^\la(\al)$ is a scalar operator for any $\la\in\Y$. It follows that 
the correspondence $\al\mapsto \Tb^\la(\al)$ leads to an embedding of the algebra  $\ZZ$ into the algebra $\C^\Y$ of functions $\Y\to\C$. We will identify $\ZZ$ with its image under this embedding. 

On the other hand, $\Sym^*$ can also be interpreted as a subalgebra of $\C^\Y$, by making use of the specialization
$$
f^*(\la)=f^*(\la_1,\la_2,\dots), \quad f^*\in\Sym^*, \quad \la=(\la_1,\la_2,\dots)\in\Y. 
$$

\begin{theorem}\label{thm3.A}
The virtual center $\ZZ$ contains the algebra $\Sym^*$, with the understanding that both algebras are regarded as subalgebras of $\C^\Y$. 
\end{theorem}

The proof is given below, after Lemma \ref{lemma3.A}. Note that $\ZZ$ actually coincides with $\Sym^*$ (see the end of the proof of Theorem \ref{thm6.A}).  

One more comment: Theorem \ref{thm3.A} shows that the virtual center is large enough in the sense that it separates the representations $T^\la$. Indeed, this follows from the fact that the algebra $\Sym^*$ separates the points of $\Y$: if $\la$ and $\mu$ are distinct Young diagrams, then one can find an element $h^*\in\Sym^*$ such that $h^*(\la)\ne h^*(\mu)$. 

To state the next lemma  we need to introduce two kinds of shifted symmetric functions indexed by $k=1,2,\dots$:
$$
\p_k(x_0,x_1,x_2,\dots):=\sum_{i=0}^\infty\left((x_i-i)^k-(-i)^k\right)
$$
(note a shift in the indexation of the variables!) and
$$
\q_k(x_1,x_2,\dots):=\sum_{i=1}^\infty\left((x_i-i)^k-(-i)^k\right).
$$
Note that
\begin{equation}\label{eq3.M}
\p_k(x_0,x_1,x_2,\dots)=x_0^k+\q_k(x_1,x_2,\dots).
\end{equation}

\begin{lemma}\label{lemma3.A}
Fix $\la\in\Y$ and suppose $n\ge |\la|+\la_1$. For any $k=1,2,\dots$ the following formula holds
\begin{equation}\label{eq3.E}
\sum_{i=0}^k(-1)^i\binom ki \frac1{n^i}\p_{k+i}(\lan)=\q_k(\la)+\q_1^k(\la)+(\cdots),
\end{equation}
where $(\cdots)$ is a linear combination of $\q_{k+1}(\la), \dots,\q_{2k}(\la)$ and $\q_1^{k+1}(\la),\dots,\q_1^{2k}(\la)$ with some coefficients of order $O(n^{-1})$. 
\end{lemma}

\begin{proof}
By \eqref{eq3.M},
$$
\p_{k+i}(\lan)=(n-|\la|)^{k+i}+\q_{k+i}(\la)=(n-\q_1(\la))^{k+i}+\q_{k+i}(\la).
$$
According to this, the left-hand side \eqref{eq3.E} can be written as the sum of two expressions: one is
\begin{equation*}
\sum_{i=0}^k(-1)^i\binom ki \frac1{n^i}\q_{k+i}(\la)=\q_k(\la)+(\cdots)
\end{equation*}
and the other is 
\begin{multline*}
\sum_{i=0}^k(-1)^i\binom ki \frac1{n^i}(n-\q_1(\la))^{k+i}=\frac{(n-\q_1(\la))^k}{n^k}\sum_{i=0}^k(-1)^i\binom ki (n-\q_1(\la))^i n^{k-i}\\
=\frac{(n-\q_1(\la))^k}{n^k}\sum_{i=0}^k\binom ki (\q_1(\la)-n)^i n^{k-i}=\frac{(n-\q_1(\la))^k}{n^k} \, \q_1^k(\la) =\q_1^k(\la)+(\cdots).
\end{multline*}
In both expressions, the rest terms have the desired form, which proves \eqref{eq3.E}. 
\end{proof}

\begin{proof}[Proof of Theorem \ref{thm3.A}]

Step 1. It suffices to show that for any  $h^*\in\SSym$ there exists a sequence of elements $\al_n\in\ZZ\C[S(n)]$ such that their degrees are uniformly bounded and 
\begin{equation}\label{eq3.D}
\lim_{n\to\infty} \wh\al_n(\lan)=h^*(\la_1,\la_2,\dots) \quad \text{\rm for any $\la\in\Y$}.
\end{equation}
Indeed, assume \eqref{eq3.D} holds true. Fix $\la\in\Y$ and recall (see Proposition \ref{prop2.B}) that $T^\la$ is equivalent to $\varinjlim \pi^\lan$. Therefore, \eqref{eq3.D} implies that the operators $T^\la(\al_n)$ converge to the scalar operator $h^*(\la_1,\la_2,\dots)\cdot 1$ on a dense subspace of $H(T^\la)$. Next, recall (see Proposition \ref{prop2.A}) that for large $n$, the restriction of $T^\la$ to $S(n)$ involves only representations of the form $\pi^{\mu[n]}$ with $\mu\subseteq\la$. Because the number of such $\mu$'s does not depend on $n$, we have a uniform bound on the norms of the operators $T^\la(\al_n)$, as it follows from \eqref{eq3.D} applied to $\mu$ instead of $\la$. Therefore, the operators $T^\la(\al_n)$ strongly converge to the scalar operator $h^*(\la_1,\la_2,\dots)\cdot1$.  Finally, since the degrees of $\al_n$'s are  uniformly bounded by the hypothesis, we conclude that the sequence $\al:=(\al_n)$ is convergent and $\Tb^\la(\al)=h^*(\la_1,\la_2,\dots)\cdot1$ for each $\la\in\Y$. Because this holds for any $h^*\in\Sym^*$, it follows that $\ZZ$ contains the algebra $\Sym^*$.  

Step 2. We proceed to the proof of \eqref{eq3.D}. It suffices to check \eqref{eq3.D}  for a system $\{h^*_k\}$ of generators of the algebra $\SSym$. Set  
\begin{equation}\label{eq3.G}
h^*_k:=\q_k+\q_1^k, \quad k=1,2,\dots,
\end{equation}
and let us show that these are generators. Indeed, the subalgebra generated by these elements contains $\q_1=h^*_1/2$. Then \eqref{eq3.G} implies that  it also contains all $\q_2,\q_3,\dots$\,. On the other hand, the elements $\q_1,\q_2,\dots$  generate the whole algebra, because the highest degree term $[\q_k]$ is $p_k$. Therefore, the elements $h^*_k$ are generators, too.  

Step 3. Thus, it remains to check \eqref{eq3.D} for the elements \eqref{eq3.G}. To do this we employ the correspondence $f^*\mapsto c_n(f^*)$ established in Corollary \ref{cor3.A}. Fix an arbitrary $k$ and set
\begin{equation}\label{eq3.I}
\al_{k,n}:=\sum_{i=0}^k(-1)^i\binom ki \frac1{n^i}c_n(\p_{k+i}).
\end{equation}
The degrees of the elements $\al_{k,n}$ admit a uniform on $n$ bound: indeed, \eqref{eq3.I} and \eqref{eq3.H2} imply $\deg \al_{k,n}\le 2k$. From \eqref{eq3.A} we see that for any fixed $\la\in\Y$,
$$
\wh \al_{k,n}(\lan)=\sum_{i=0}^k(-1)^i\binom ki \frac1{n^i}\p_{k+i}(\lan).
$$
Finally, by Lemma \ref{lemma3.A}, the limit of this quantity as $n\to\infty$ equals $\q_k(\la)+\q_1^k(\la)=h^*_k(\la)$, as desired. 

This completes the proof. 
\end{proof}

\section{The semigroup method}\label{sect4}

\subsection{Generalities about semigroup representations}

Let $\Ga$ be a semigroup with the unity (that is, a monoid), equipped with an involution $\ga\mapsto\ga^*$ (an involutive anti-automorphism). 

Let $H$ be a complex Hilbert space of finite or countable dimension. A \emph{contraction} on $H$ is an operator with norm $\le1$. Let $C(H)$ be the set of contractions. It is a semigroup with involution (the conventional conjugation of bounded operators). 

\begin{definition}\label{def4.A}
By a \emph{representation} of $\Ga$ on $H$ we mean a homomorphism $\T: \Ga\to C(H)$ which preserves the unity and is compatible with the involution, that is,  $\T(\ga^*)=(\T(\ga))^*$ for all $\ga\in\Ga$. 
\end{definition}

Any such $\T$  is completely reducible. If $\Ga$ is finite, then $\T$ can be decomposed into a direct sum of irreducible representations, and any irreducible representation has finite dimension. All these facts are evident.

\subsection{The semigroups $\Ga(n)$}\label{sect4.2}

For $n=1,2,\dots$, we denote by $\Ga(n)$ the set of $n\times n$ matrices with the entries in $\{0,1\}$ and such that that each row and each column contains at most one 1. This is a semigroup under matrix multiplication, a basic example of an \emph{inverse semigroup} \cite{CP61-67}. The semigroups $\Ga(n)$ are known under the name of (finite) \emph{symmetric inverse semigroups} or else \emph{rook monoids}, see \cite{CP61-67}, \cite{Sol02}. 

The matrix transposition determines an involution on $\Ga(n)$. The semigroup $\Ga(n)$ contains the unity $e=1$ (the diagonal matrix with all diagonal entries equal to $1$) and the zero $\emptyset$ (the zero matrix). The group of invertible elements in $\Ga(n)$ is the symmetric group $S(n)$. 

Recall that $\N_n:=\{1,2,\dots,n\}$. Given a nonempty subset $I\subseteq \N_n$, the element $\eps_I\in\Ga(n)$ is the diagonal matrix with the entries
\begin{equation}\label{eq4.A}
(\eps_I)_{ii}:=\begin{cases} 0, & i\in I, \\ 1, & i\in \N_n\setminus I. \end{cases}
\end{equation}
Note that $\eps_I$ is a selfadjoint idempotent:
$$
\eps_I^2=\eps_I^*=\eps_I.
$$
For $i\in \N_n$, we write $\eps_i$ instead of $\eps_{\{i\}}$. 

The semigroup $\Ga(n)$ is generated by $S(n)$ together with the idempotents $\eps_1,\dots,\eps_n$ (actually, it suffices to keep a single such idempotent). 

\begin{proposition}[\cite{MO01}, Proposition 3.1]\label{prop4.B}
The semigroup algebra $\C[\Ga(n)]$ is generated by the elementary transpositions $s_1,\dots,s_{n-1}$ and the elements $\eps_1,\dots,\eps_n$. They satisfy the relations
\begin{gather*}
\eps_i^2=\eps_i, \quad \eps_i\eps_j=\eps_j\eps_i, \quad i,j=1,\dots,n;\\
s_i\eps_i=\eps_{i+1}s_i, \quad s_i\eps_i\eps_{i+1}=\eps_i\eps_{i+1}, \quad i=1,\dots,n-1.
\end{gather*}
Furthermore, together with the Coxeter relations for the elementary transpositions, these are defining relations in the algebra $\C[\Ga(n)]$
\end{proposition}

Note that the specialization map $\eps_1\to0,\dots,\eps_n\to0$ determines a surjective algebra morphism $\C[\Ga(n)]\to\C[S(n)]$.

\subsection{Representations of $\Ga(n)$}\label{sect4.3}

Given a Young diagram $\la$ with $|\la|\le n$, we construct a representation $\T^\la_n$ of $\Ga(n)$ as follows. 

First of all, if $\la=\varnothing$, then $\T^\varnothing_n$ is the one-dimensional representation such that $\T^\varnothing_n(\ga)=1$ for all $\ga\in\Ga(n)$. We call it the \emph{trivial representation}; it is characterized by the property that the image of $\emptyset\in\Ga(n)$ is $1$.  

Assume now $|\la|=\ell$ with $1\le\ell\le n$, and denote by $\Om(\ell,n)$ the set of $\ell\times n$ matrices $\om$ with the entries in $\{0,1\}$ and such that each row contains exactly one $1$ and each column contains at most one $1$ (cf. the definition of $\Om(\ell,\infty)$ in section \ref{sect2.2}). Note that the rank of $\om$ equals $\ell$.

Take the irreducible representation $\pi^\la$ of the group $S(\ell)$ indexed by $\la$, and consider the space of $H(\pi^\la)$-valued vector functions on $\Om(\ell,n)$ with the scalar product given by the formula \eqref{eq2.K}, where we only replace $\Om(\ell,\infty)$ by $\Om(\ell,n)$. Then  we define $H(\T^\la_n)$ as the subspace of functions satisfying the condition 
$$
\phi(s\om)=\pi^\la(s)\phi(\om), \quad s\in S(\ell), \; \om\in \Om(\ell,n),
$$
cf. \eqref{eq2.L}. 

Next, given $\ga\in\Ga(n)$, we define an operator $\T^\la_n(\ga)$ on $H(\T^\la_n)$ as follows (cf. \eqref{eq2.A}:
\begin{equation}\label{eq4.B}
(\T^\la_n(\ga)\phi)(\om)=\begin{cases} \phi(\om\ga), & \text{if $\om\ga\in\Om(\ell,n)$}, \\
0, & \text{otherwise}.
\end{cases}
\end{equation}

\begin{proposition}\label{prop4.D}

Let $\la$ range over the set of nonempty Young diagrams with at most $n$ boxes. 

{\rm(i)} For each such $\la$, the operators \eqref{eq4.B} form a representation $\T^\la_n$ of the semigroup $\Ga(n)$.

{\rm(ii)} The representations $\T^\la_n$ are irreducible and pairwise non-equivalent.

{\rm(iii)} These representations, together with $\T^\varnothing_n$, exhaust all irreducible representations of the semigroup $\Ga(n)$.

\end{proposition}

\begin{proof}
See \cite[sect. 6]{Ols85} or \cite[Theorem 2.9]{MO01}. The proposition can also be extracted from the papers of Munn \cite{Mun57}, \cite{Mun78}, and Solomon \cite[sect. 2]{Sol02}.
\end{proof}

\subsection{The semigroup algebra of $\Ga(n)$}

Let $\C[\Ga(n)]$ denote the semigroup algebra of $\Ga(n)$. We endow it with the involution inherited from $\Ga(n)$. Any representation of $\Ga(n)$ (in the sense specified above) canonically extends to a representation of the involutive algebra $\C[\Ga(n)]$. The next proposition describes the structure of $\C[\Ga(n)]$

\begin{proposition}\label{prop4.A}
Endow the semigroup algebra $\C[\Ga(n)]$ with the involution inherited from $\Ga(n)$. We have 
$$
\C[\Ga(n)]\simeq \bigoplus_{\la: \,|\la|\le n} \End(H(\T^\la_n)), 
$$ 
isomorphism of $\ast$-algebras.
\end{proposition}

\begin{proof}
There is a natural algebra morphism from $\C[\Ga(n)]$ to $\bigoplus \End(H(\T^\la_n))$. Because the representations $\T^\la_n$ are irreducible and pairwise nonequivalent, it is surjective. From the definition of $\Ga(n)$ it follows that the dimension of the first algebra is
$$
\dim \C[\Ga(n)]=|\Ga(n)|=\sum_{\ell=0}^n\binom n\ell^2 \ell!.
$$
On the other hand, from the description of $H_n(T^\la)$ (see the remark after \eqref{eq2.A}) it follows that
$$
\dim\T^\la_n=\dim H_n(T^\la)=\binom{n}{|\la|}\dim\pi^\la.
$$
Using this  we obtain that the dimension of the second algebra is the same:
$$
\sum_{\la: \,|\la|\le n} (\dim\T^\la_n)^2=\sum_{\la: \,|\la|\le n}  \binom n{|\la|}^2(\dim\pi^\la)^2 =\sum_{\ell=0}^n \binom n\ell^2\sum_{\la\in\Y_\ell}(\dim\pi^\la)^2=\sum_{\ell=0}^n \binom n\ell^2 \ell!.
$$
Therefore, the morphism is also injective.
\end{proof}

\begin{remark}
Proposition \ref{prop4.A} implies that $\C[\Ga(n)]$ is a finite-dimensional $C^*$-algebra, as in the finite group case. The same holds true for any finite inverse semigroup, as it can be realized as a $\ast$-subsemigroup of $\Ga(n)$ for $n$ large enough (a well-known theorem). Conversely, one can prove that if the semigroup algebra of a finite $\ast$-semigroup is a $C^*$-algebra, then it admits a faithful $\ast$-representation, which implies that the semigroup is inverse, see \cite[sect. 4]{Ols91b}. 
\end{remark}

\subsection{The semigroup $\Ga(\N)$}

Recall that $\N$ is our notation for the set of $\{1,2,\dots\}$ of positive integers. Let $\Ga(\N)$ be the set of matrices of the format $\N\times\N$, with the entries in $\{0,1\}$ and such that each row or column contains at most one $1$. Thus, the definition is similar to that of $\Ga(\N)$, only the matrices become now infinite. Like $\Ga(n)$, the set $\Ga(\N)$ forms a semigroup with respect to matrix multiplication. The involution is defined in the same way (the matrix transposition). The semigroup $\Ga(\N)$ contains the symmetric group $S(\N)$ (recall that this is our notation for the group of all permutations of $\N$). The group $S(\N)$ is exactly the subset of invertible elements in $\Ga(\N)$. In particular, $\Ga(\N)$ contains the group $S(\infty)$.  

By making use of  the natural embedding $\Ga(\N)\hookrightarrow\{0,1\}^{\N\times \N}$ we equip $\Ga(\N)$ with the  topology induced by the product topology of $\{0,1\}^{\N\times \N}$. With respect to it, $\Ga(\N)$ is a compact \emph{semitopological} semigroup: the multiplication map is separately continuous but not jointly continuous. 

\subsection{The truncation maps}

Given a matrix $\ga\in\Ga(\N)$, we denote by $\th_n(\ga)$ its $n\times n$ upper-left corner, where $n=1,2,\dots$\,. We call $\th_n$ the $n$th \emph{truncation map}. It maps $\Ga(\N)$ onto $\Ga(n)$.  

A simple but important fact is that the restriction of $\th_n$ to $S(\infty)\subset\Ga(\N)$ is surjective, too. Even more, $\th_n(S(2n))=\Ga(n)$. 

This implies, in particular, that $S(\infty)$ is dense in $\Ga(\N)$. Thus, $\Ga(\N)$ is a certain compactification of $S(\infty)$.  

Note also that the topology on $S(\N)$ induced from $\Ga(\N)$ coincides with the topology introduced in section \ref{sect1.1}.

\subsection{Link with tame representations of $S(\infty)$}\label{sect4.7}

Given a subset $I\subseteq\N$, let $\eps_I\in\Ga(\N)$ stand for the diagonal matrix with the entries 
\begin{equation}\label{eq4.A1}
(\eps_I)_{ii}:=\begin{cases} 0, & i\in I, \\ 1, & i\in \N\setminus I. \end{cases}
\end{equation}
This definition is similar to \eqref{eq4.A}. 

\begin{proposition}\label{prop4.C}

{\rm(i)} If $\T$ is a representation of\/ $\Ga(\N)$, then its restriction $T$ to $S(\infty)$ is a tame representation of this group.  Conversely, any tame representation $T$ of the group $S(\infty)$ admits a canonical extension by continuity to a representation $\T$ of\/ $\Ga(\N)$. 

{\rm(ii)} Let $T$ be a tame representation of $S(\infty)$ and $\T$ be its canonical extension to\/ $\Ga(\N)$. For arbitrary $n=0,1,2,\dots$, the projection operator to the subspace $H_n(T)$ is given by the operator $\T(\eps_{\N\setminus \N_n})$.

{\rm(iii)} Let $T$ and $\T$ be as above, and let $n$ be so large that $H_n(T)\ne\{0\}$. Then $\T$ gives rise to a representation $\T_n$ of the semigroup $\Ga(n)$ on the space $H_n(T)$. This representation is uniquely  defined by the relation
$$
\T_n(\th_n(\ga))=\T(\eps_{\N\setminus\N_n}\ga\,\eps_{\N\setminus\N_n})\big|_{H_n(T)}, \quad \ga\in\Ga(\N).
$$
\end{proposition}

We say that $\T_n$ is \emph{associated with $T$}.

\begin{proof}
See \cite[section 5]{Ols85}. Note a discrepancy in notation: in \cite{Ols85}, the semigroup $\Ga(\N)$ is denoted by $\Ga(\infty)$, whereas we reserve the symbol $\Ga(\infty)$ for denoting a smaller semigroup, see section \ref{sect4.8} below. 
\end{proof}

For the representations $T^\la$, the above assertions can be verified directly from their realization (see section \ref{sect2.2}).  The case $\la=\varnothing$ is trivial, so we assume $\la\ne\varnothing$ and set $\ell:=|\la|$. The extension of $T^\la$ to $\Ga(\N)$ will be denoted by $\T^\la$; it is given by the following modification of formula \eqref{eq2.A}: for $\ga\in\Ga(\N)$, 
\begin{equation}\label{eq4.C}
(\T^\la(\ga) \phi)(\om)=\begin{cases} \phi(\om \ga), &\text{if $\om\ga\in\Om(\ell,\infty)$}, \\
0, & \text{otherwise}.
\end{cases}
\end{equation}

\begin{remark}\label{rem4.A}
Recall that the subspace  $H_n(T^\la)$ is non-null if $n\ge \ell$; then the associated representation of the semigroup $\Ga(n)$ on this space is the representation $\T^\la_n$ defined in \eqref{eq4.B}. 
\end{remark}

\subsection{The semigroup $\Ga(\infty)$}\label{sect4.8}

For any $n=1,2,3,\dots$, we define the \emph{canonical embedding} $\Ga(n)\hookrightarrow\Ga(n+1)$ as the map
$$
\ga\mapsto \begin{bmatrix} \ga & 0\\ 0 & 1\end{bmatrix}, \quad \ga\in\Ga(n).
$$ 
With the use of these embeddings, we set $\Ga(\infty):=\varinjlim \Ga(n)$. 

Equivalently, $\Ga(\infty)$ can be defined as the subsemigroup of $\Ga(\N)$ formed by the matrices $\ga\in\Ga(\N)$ subject to the condition $\ga_{ij}=\de_{ij}$ as $i+j$ gets large enough. 

Thus, we have 
$$
S(\infty)\subset \Ga(\infty)\subset \Ga(\N).
$$

\section{The centralizer construction}\label{sect5}

\subsection{The algebras $\A_m$ and $\A$}\label{sect5.1}

Given a matrix $\ga\in\Ga(n)$, we define its \emph{degree}, $\deg (\ga)$, as the number of $0$'s on the diagonal.  On $S(n)\subset\Ga(n)$, this agrees with the definition given in section \ref{sect3.3}. 

We denote by $\A(n)$ the semigroup algebra $\mathbb{C}[\Gamma(n)]$ and endow it with the ascending filtration induced by $\deg(\,\cdot\,)$. That is, the $k$th term of the filtration is spanned by the semigroup elements with degree at most $k$. Note that (\cite[Proposition 3.4]{MO01})
$$
\deg(\ga_1\ga_2)\le\deg(\ga_1)+\deg(\ga_2), \qquad \ga_1,\ga_2\in\A(n),
$$
so $\A(n)$ becomes a filtered algebra.

The projection
$$
\theta_{n,n-1}: \Gamma(n)\rightarrow \Gamma(n-1), \quad \ga\mapsto\theta_{n,n-1}(\ga),
$$
consists in striking from $\ga$ its last row and column. We extend $\theta_{n,n-1}$  to a linear map $\theta_{n,n-1}: \A(n)\rightarrow\A(n-1)$. Note that
$$
\deg(\theta_{n,n-1}(\gamma))\le \deg(\gamma).
$$

\begin{definition}\label{def5.A}
Let $m=0,\dots,n$.

(i) We denote by $\Gamma_m(n)\subseteq \Ga(n)$ the subsemigroup of elements with first $m$ diagonal entries equal to $1$. Note that  $\Ga_m(n)$ is isomorphic to $\Ga(n-m)$ (by definition, $\Ga(0)$ consists of a single element, which is automatically the unit). 

(ii) We denote by $\A_m(n)\subseteq\A(n)$ the centralizer of $\Ga_m(n)$ in $\A(n)$. Note that $\A_0(n)$ is the center of $\A(n)$ and $\A_n(n)=\A(n)$.
\end{definition}

As shown in \cite[Proposition~3.3]{MO01}, the restriction of $\theta_{n,n-1}$ to $\A_{n-1}(n)\subset \A(n)$ defines a unital algebra morphism
$\A_{n-1}(n)\rightarrow \A(n-1)$; moreover,
$$
\theta_{n,n-1}(\A_m(n))\subset \A_m(n-1)\quad \text{for}\quad 0\le m \le n-1.
$$
Thus,  for any $m\ge 0$ we have a projective system of filtered algebras
$$
\A_m(m)\leftarrow \A_m(m+1) \leftarrow \cdots \leftarrow \A_m(n) \leftarrow \cdots\,.
$$

\begin{definition}[The centralizer construction]\label{def5.B}

Let $m=0,1,2,\dots\,.$

(i) We set
$$
\A_m := \varprojlim\A_m(n), \quad n\to\infty,
$$
where the projective limit is taken in the category of filtered algebras.  
In more detail, each element $a\in\A_m$ is a sequence $(a_n)_{n\ge m}$ such that
$$
a_n\in \A_m(n),\quad a_{n-1} = \theta_{n,n-1}(a_n), \quad \deg(a) := \sup \deg(a_n)<\infty.
$$

(ii) Because $\A_m\subset \A_{m+1}$,  the following definition makes sense 
$$
\A := \varinjlim \A_m, \quad m\to \infty.
$$
\end{definition}

Note that there is a natural embedding $\C[\Ga(\infty)]\hookrightarrow\A$ under which $\C[\Ga(m)]$ is mapped into $\A_m$ for each $m$ (\cite[Proposition 3.7]{MO01}). The image of $\C[\Ga(\infty)]$ consists of stable sequences $a=(a_n)$, where `stable' means that $a_n$ does not depend  on $n$ provided $n$ gets large enough.
In particular, the image of an element $\ga\in S(\infty)$ is represented by the sequence $(a_n\equiv \ga)\in\A_m$, where $m$ is chosen so large that $\ga\in \Ga(m)$. 

\subsection{Structure of the algebra $\A_0$}\label{sect5.2}

By the very definition, the algebra $\A_0$ is contained in the center of $\A$. Actually, $\A_0$ is the whole center  (\cite[Proposition 3.9]{MO01}). 

As shown in \cite[Corollary~4.11]{MO01}, the algebra $\A_0$ has a countable system of algebraically independent generators $\{\De^{(k)}\}$, $k=1,2,\dots$\,. They are constructed as follows. 

According to the general definition, each $\De^{(k)}$ is determined by a sequence $(\De^{(k)}_n\in\A_0(n))$, where $n=1,2,\dots$. Thus, to define $\De^{(k)}$, one has to specify the elements $\De^{(k)}_n$.

For $k=1$ one has 
$$
\De^{(1)}_n:=\epsb_1+\dots+\epsb_n,
$$
where we use the shorthand notation
$$
\epsb_i:=1-\eps_i, \quad i=1,2,\dots
$$
(note that the elements $\epsb_i$ are idempotents).  

Next, for $k\ge2$, one has (cf. \eqref{eq3.F1} and \eqref{eq3.F2}) 
$$
\De^{(k)}_n:=0 \quad \text{for $n<k$}
$$
and
\begin{equation}\label{eq5.H}
\De^{(k)}_n:=\sum (i_1, \dots, i_k)\epsb_{i_1}\dots\epsb_{i_k} \quad  \text{for $n\ge k$},
\end{equation}
where the sum is taken over all ordered $k$-tuples $i_1, \ldots, i_k\in\N_n$ of pairwise distinct numbers. For instance, 
\begin{equation}\label{eq5.F}
\De^{(k)}_2:=\sum_{1\le i\ne j\le n}(i,j)\epsb_i\epsb_j.
\end{equation}

\subsection{Structure of the algebras $\A_m$, $m\ge1$}\label{sect5.3}

The algebra $\A$ contains elements $u_1,u_2, \dots$, where the $i$th element is represented by the following sequence $(u_{i\mid n}\in\A(n): n>i)$:  
\begin{equation}\label{eq5.G}
u_{i\mid n}:=\sum_{j=i+1}^n (i,j) \epsb_i\epsb_j=\sum_{j=i+1}^n \epsb_i\epsb_j(i,j) .
\end{equation}
This is an analog of the Jucys--Murphy elements (see \cite[section 3.2.1]{CST10}). It is easily verified that if $i\le m<n$, then $u_{i\mid n}\in\A_m(n)$. It follows that $u_i\in\A_m$ for $i=1,\dots,m$.

Observe also that the algebra $\A_m$ contains the semigroup algebra $\C[\Ga(m)]$. Indeed, this follows from the fact any element of $\C[\Ga(m)]$ commutes with $\Ga_m(n)$, for any $n>m$.

\begin{definition}\label{def5.C}

Let $\wt\HH_m\subset\A_m$ be the subalgebra generated by $\Ga(m)$ and $u_1,\dots,u_m$. Equivalently, $\wt\HH_m\subset\A_m$ is generated by $S(m)$, $\eps_1,\dots,\eps_m$, and $u_1,\dots,u_m$. 
\end{definition}

Recall that $s_1,\dots,s_{m-1}$ are the elementary transpositions in $S(m)$. 

\begin{lemma}[\cite{MO01}, Proposition 5.17]\label{lemma5.A}
The following commutation relations hold in the algebra $\wt\HH_m$
\begin{gather*}
s_i u_i=u_{i+1}s_i + \epsb_i\epsb_{i+1}, \quad i=1,\dots,m-1; \\
s_i u_j=u_j s_i, \quad j\ne i, i+1; \\
u_iu_j=u_ju_i, \quad \eps_iu_i=u_i\eps_i=0, \quad i,j=1,\dots, m; \\
\eps_i u_j=u_j \eps_i, \quad i\ne j.
\end{gather*}
\end{lemma}

Note that the specialization map $\eps_1\to0,\dots,\eps_m\to0$ determines a surjective algebra morphism $\wt\HH_m\to\HH_m$, where $\HH_m$ is the degenerate affine Hecke algebra associated with the symmetric group $S(m)$.

\begin{proposition}[\cite{MO01}, Theorem 5.18] \label{prop5.B}

Fix $m\ge1$.

{\rm(i)} For each $m\in\Z_{\ge1}$, the natural morphism $\A_0\otimes\wt\HH_m\to \A_m$ induced by the multiplication map is an algebra isomorphism.  

{\rm(ii)} The  algebra $\wt\HH_m$ is isomorphic to the abstract algebra generated by the semigroup algebra $\A(m)=\C[S(m)]$ and the elements $u_1,\dots,u_m$ subject to the commutation relations from Lemma \ref{lemma5.A}.
\end{proposition}

\subsection{The use of endomorphism $\xi$ of the algebra $\A$}\label{sect5.4}

Recall that we have defined in \eqref{eq2.C} an endomorphism $\xi$ of the group $S(\infty)$. Here we introduce a related endomorphism of the algebra $\A$, also denoted by $\xi$. 

Given $n=1,2,\dots$, consider the injective semigroup morphism $\Ga(n)\to \Ga(n+1)$ assigning to a matrix $M$ the matrix $\begin{bmatrix} 1 &0\\0&M\end{bmatrix}$. This `shift morphism' induces an algebra embedding $\xi_{n,n+1}:\A(n)\to\A(n+1)$ preserving the unit. Note that the diagram 
$$
\begin{CD}
\A(n-1) @<\th_{n,n-1}<< \A(n)\\
@V\xi_{n-1,n}VV @VV\xi_{n,n+1}V\\
\A(n) @<\th_{n+1,n}<< \A(n+1)
\end{CD}
$$
is commutative: the operations of shift and truncation are permutable. Furthermore, $\xi_{n,n+1}$  preserves the filtration and 
$$
\xi_{n,n+1}(\A_m(n))\subset\A_{m+1}(n+1) \quad \text{for each $m\le n$}. 
$$
Therefore, applying the `shifts' $\xi_{n,n+1}$ component-wise to $a=(a_n)\in\A$, we obtain an injective endomorphism $a\mapsto \xi(a)$ of the algebra $\A$, which shifts $\A_m$ into $\A_{m+1}$ for each $m\ge0$. 

Below $\xi^i$ denotes the $i$th iteration of $\xi$; here $i=0,1,2,\dots$ and $\xi^0:=\id$. 

\begin{lemma}\label{lemma5.B}
One has
\begin{equation}\label{eq5.A}
2u_i=\xi^{i-1}(\De^{(2)})-\xi^i(\De^{(2)}), \quad i=1,2,\dots\,.
\end{equation}
\end{lemma}

\begin{proof} 
From \eqref{eq5.F} we obtain 
\begin{equation*}
\De^{(2)}_{n+1}-\xi_{n,n+1}(\De^{(2)}_n)=2u_{1\mid n+1},
\end{equation*}
which proves \eqref{eq5.A} for the particular case $i=1$:
\begin{equation}\label{eq5.C}
2u_1=\De^{(2)}-\xi(\De^{(2)}).
\end{equation}

Next, observe that 
$$
\xi_{n,n+1}(u_{i\mid n})=u_{i+1\mid n+1}.
$$
From this it follows that $\xi(u_i)=u_{i+1}$. Using this relation and applying $\xi^{i-1}$ to both sides of \eqref{eq5.C} we finally obtain \eqref{eq5.A} for $i\ge2$. 
\end{proof}

Combining the lemma with the results of section \ref{sect5.3} we obtain 
\begin{corollary}\label{cor5.A}
For each $m\ge1$, the algebra $\A_m$ is generated by the subalgebras 
\begin{equation}\label{eq5.B}
\A_0, \; \xi(\A_0),\; \xi^2(\A_0),\;\dots, \;\xi^m(\A_0),
\end{equation} 
together with the group $S(m)$ and the elements $\eps_1,\dots,\eps_m$.
\end{corollary}

We use this fact in the proof of Proposition \ref{prop5.A} and Theorem \ref{thm6.A}.

\subsection{Extension of tame representations to the algebra $\A$}\label{sect5.5}

Let $T$ be a tame representation. As usual, we denote by $\T_n$ the associated representation of $\Ga(n)$ acting on $H_n(T)$, where $n\ge\cond(T)$ (recall that $\cond(T)$ is the conductor of $T$, see Definition \ref{def2.A}, item (ii)). 

As shown in \cite[Proposition 3.10]{MO01}, for each $a\in\A$ there exists an operator  $\wt\T(a)$ acting on the dense subspace $H_\infty(T)$, which is uniquely determined by the following property: pick $m$ so large that $a\in\A_m$ and write $a=(a_n: n\ge m)$. Then for any $n\ge\max(m,\cond(T))$ one has
\begin{equation}\label{eq5.D}
\wt\T(a)\big|_{H_n(T)}=\T_n(a_n).
\end{equation}
Furthermore, the correspondence $a\mapsto\wt\T(a)$ determines a representation of the algebra $\A$ on the space $H_\infty(T)$. 

If $T$ is irreducible, i.e. $T=T^\la$ for some $\la\in\Y$, then the representation $\wt\T^\la$ on $H_\infty(T)$ is irreducible in the purely algebraic sense, which implies that the center $\A_0$ acts on $H_\infty(T^\la)$ by scalar operators (see \cite[Proposition 3.11]{MO01}). Thus, 
\begin{equation}\label{eq5.E}
\wt\T^\la(a)=\wh a(\la)\cdot 1, \quad a\in\A_0,
\end{equation}
where $\wh a(\la)$ denotes a scalar depending on $\la$ and `$1$' is the identity operator on $H_\infty(T^\la)$. 

Note that if $n\ge\cond(T^\la)=|\la|$ and $a=(a_m)\in\A_0$, then $\T^\la_n(a_n)$ acts on $H_n(T^\la)$ as the operator of multiplication by $\wh a(\la)$.

\begin{proposition}\label{prop5.C}
For the central elements $\De^{(k)}$ introduced in section \ref{sect5.2}, one has
\begin{equation*}
\wh\De^{(k)}(\la)=p_k^\#(\la_1,\la_2,\dots), \quad  k=1,2,\dots, \quad \la\in\Y, 
\end{equation*}
where the elements $p^\#_k=p^\#_{(k)}\in\Sym^*$ were defined in the end of section \ref{sect3.1}.
\end{proposition}

\begin{proof} This is shown in \cite[Proposition 4.20 and the subsequent remark]{MO01}. In section \ref{sect10} we give a detailed proof in the more general situation, for wreath products; see Proposition \ref{prop10.C}.
\end{proof} 

This implies

\begin{corollary}[\cite{MO01}, Theorem 4.21]\label{cor5.B}
The correspondence $a\mapsto \wh a(\,\cdot\,)$ determined by \eqref{eq5.E} establishes an isomorphism $\A_0\to \Sym^*$ of filtered algebras.
\end{corollary}

We also need the following proposition.

\begin{proposition}\label{prop5.A}
Let $\la\in\Y$. For any $a\in\A$, the corresponding operator $\wt\T^\la(a)$, which is initially defined on the dense subspace $H_\infty(T^\la)$, is bounded and hence can be extended to the whole Hilbert space $H(T^\la)$. 
\end{proposition}

\begin{proof}
Because $\A$ is the union of the subalgebras $\A_m$, it suffices to prove the claim for a system of generators of $\A_m$, for each $m$. A system of generators is provided by Corollary \ref{cor5.A}: this is the group $S(m)$, the elements $\eps_1,\dots,\eps_m$, and the subalgebras \eqref{eq5.B}. 

The elements of $S(m)$ act as unitary operators. The elements $\eps_1,\dots,\eps_m$ are selfadjoint idempotents and are represented by selfadjoint projection operators. Thus, it remains to prove the boundedness of the operators coming from the subalgebras \eqref{eq5.B}.

This is evident for the subalgebra $\A_0$, because it acts by scalar operators. Next, from Corollary \ref{cor2.A} it follows that there exists an orthogonal basis in $H(T^\la)$, contained in $H_\infty(T^\la)$ and possessing the following property: for each $k\ge1$ and any $a\in\xi^k(A_0)$, the operator $\wt\T(a)$ is diagonalized in this basis and has only finitely many distinct eigenvalues. This implies that $\wt\T(a)$ is bounded. 
\end{proof}

\section{Main theorem, case of the infinite symmetric group $S(\infty)$}\label{sect6}

\subsection{The embedding $\A\to \prod_{\la\in\Y}\End(H(T^\la))$}\label{sect6.1}

For a Hilbert space $H$, we denote by $\End(H)$ the algebra of bounded linear operators on $H$. Introduce the countably infinite product space
$$
\mathcal E:=\prod_{\la\in\Y}\End(H(T^\la));
$$
it is an algebra with respect to component-wise operations. By virtue of Proposition \ref{prop5.A},  the correspondence $a\mapsto \wt\T^\la(a)$ is an algebra morphism $\A\to\End(T^\la)$ for each $\la\in\Y$. The total collection of these morphisms gives rise to an algebra morphism 
\begin{equation}\label{eq6.A}
\iota_\A: \A\to  \mathcal E, \quad a\mapsto \prod_{\la\in\Y}\wt\T^\la(a).
\end{equation}

\begin{lemma}\label{lemma6.A}
$\iota_\A$ is injective.
\end{lemma}

\begin{proof}
Let $a\in\A$ be a nonzero element. Pick $m$ so large that $a\in\A_m$ and write $a=(a_n: n\ge m)$, as usual. Since $a\ne0$, we have $a_n\ne0$ for all $n$ large enough. Fix one such $n$. By Proposition \ref{prop4.A}, there exists $\la$ with $|\la|\le n$, such that $\T^\la_n(a_n)\ne0$. By virtue of \eqref{eq5.D}, it follows that $\wt\T^\la(a)\ne0$, so that $\iota_\A(a)\ne0$. 
\end{proof}

\subsection{The algebra $\B$ --- the virtual group algebra of $S(\infty)$}\label{sect6.2}

In section \ref{sect3.4}, we have defined the virtual center $\ZZ$ of the group algebra $\C[S(\infty)]$. Now we extend that construction and introduce a larger algebra $\B\supset\ZZ$.  
 
For $0\le m\le n$ we denote by  $\B_m(n)$ the centralizer of $S_m(n)$ in $\C[S(n)]$. In particular, $\B_0(n)=\ZZ\C[S(n)]$. 

\begin{definition}\label{def6.A}

(i) Let  $\al=(\al_1,\al_2,\dots)$ be an infinite sequence of elements of $\C[S(\infty)]$ such that $\al_n\in\C[S(n)]$ for all $n$ large enough. We say that $\al$ is \emph{convergent} if the following three conditions are satisfied:

\begin{itemize}
\item for any irreducible tame representation $T^\la$, the operators $T^\la(\al_n)$  have a limit,  $\Tb^\la(\al)$, in the strong operator topology; 
\item there exists $m$ such that $\al_n\in\B_m(n)$ for all $n$ large enough;
\item $\sup \deg \al_n<\infty$. 
\end{itemize}

(ii) Next, we say that two convergent sequences, $\al$ and $\be$, are \emph{equivalent} if $\Tb^\la(\al)=\Tb^\la(\be)$ for all $\la\in\Y$. 
\end{definition}

\begin{definition}\label{def6.B}
Let $\B$ be the set of equivalence classes of convergent sequences. We endow it with a natural structure of associative algebra, induced by component-wise operations on convergent sequences (the correctness of this definition is verified as in Lemma \ref{lemma3.B}). 
\end{definition}

For each $\la$, the correspondence $\al\mapsto \Tb^\la(\al)$ determines a representation of algebra $\B$ by bounded operators on the space $H(T^\la)$: we apply the same argument as in section \ref{sect3.4}.  

\begin{lemma}\label{lemma6.B}
There is a natural embedding $\C[S(\infty)]\to\B$.
\end{lemma}

\begin{proof}
We say that a sequence $\al=(\al_n)$ is \emph{stable} if  $\al_n$ does not depend on $n$ for all $n$ large enough. Evidently, such a sequence is convergent, the equivalence classes of stable sequences form a subalgebra of $\B$,  and  $\C[S(\infty]$ is mapped, in a natural way onto this subalgebra. In this way we obtain an algebra morphism $\C[S(\infty)]\to\B$.

Let us show that it is injective. This amount to verifying that if $\al\in\C[S(\infty)]$ is non-null, then $T^\la(\al)\ne0$ for some $\la\in\Y$. A way to see this is to apply Lemma \ref{lemma6.A} combined with the fact that $C[S(\infty)]$ is contained in $\A$ and $T^\la(\al)=\wt\T^\la(\al)$ for $\al\in\C[S(\infty)]$. 
\end{proof}

Thus, $\B$ is an extension of the group algebra $\C[S(\infty)]$; we call $\B$ the \emph{virtual group algebra of $S(\infty)$}.

\subsection{The isomorphism $\B\to\A$}

For each $m=0,1,2,\dots$, we denote by $\B_m$ the subalgebra of $\B$ formed by the convergent sequences $(\al_n)$ such that $\al_n\in\B_m(n)$ for all $n$ large enough. By the very definition of the algebra $\B$, we have an algebra embedding (cf. \eqref{eq6.A})
\begin{equation*}
\iota_\B: \B\to  \mathcal E, \quad \al\mapsto \prod_{\la\in\Y}\Tb^\la(\al).
\end{equation*}

\begin{theorem}[Main Theorem]\label{thm6.A}
There exists an algebra isomorphism $\F:\B\to\A$, uniquely determined by the property that $\iota_\A\circ\F=\iota_\B$, meaning that  the diagram 
$$
  \xymatrix{ \B \ar[r]^\F \ar[dr]_{\iota_\B}& \A \ar[d]^{\iota_\A} \\
  & \mathcal E}
$$
is commutative.

Furthermore, $\F(\B_m)=\A_m$ for each $m=0,1,2,\dots$\,. 
\end{theorem}

\begin{proof}
Given a tame representation $T$, we denote by $P_\rr$ the operator of orthogonal projection $H(T)\to H_\rr(T)$, for each $\rr=0,1,2,\dots$\,. Next, we agree to identify the operators on the subspace $H_\rr(T)$ with those operators on the whole space $H(T)$ that vanish on the orthogonal complement $H(T)\ominus H_\rr(T)$. Thus, we do not distinguish between an operator $X\in\End(H(T))$, such that $X=P_\rr X P_\rr$, and its restriction to $H_\rr(T)$.  This convention will simplify some formulas below. Further, let $\T$ denote the canonical extension of $T$ to the semigroup $\Ga(\N)$, and let $\T_\rr$ (for arbitrary $\rr$ greater or equal the conductor of $T$) denote the representation of $\Ga(\rr)$ on $H_\rr(T)$ associated with $T$. 

The proof is divided into several steps.

Step 1. Fix an element $\al\in\B$ and pick a convergent system $(\al_n)$ representing $\al$. We claim that for each $\rr\ge1$ there exists a limit 
\begin{equation}\label{eq6.B}
b_\rr:=\lim_{n\to\infty} \th_\rr(\al_n),
\end{equation}
which does not depend on the choice of $(\al_n)$. The limit in \eqref{eq6.B} is understood as convergence in the finite-dimensional algebra $\A(r)=\C[\Ga(\rr)]$.

Furthermore, we claim that the element $b_\rr\in\C[\Ga(\rr)]$ is uniquely characterized by the property that  for any irreducible tame representation $T=T^\la$ with $|\la|\le\rr$ one has
\begin{equation}\label{eq6.C}
\T_\rr(b_\rr)=P_\rr \Tb(\al) P_\rr.
\end{equation}

Indeed,  we have $\T_\rr=\T^\la_\rr$ (see Remark \ref{rem4.A}). Next, by virtue of Proposition \ref{prop4.A},  to prove the convergence in \eqref{eq6.B}, it suffices to check the existence of a limit 
$$
\lim_{n\to\infty} \T_\rr (\th_\rr(\al_n))\in\End(H_\rr(T))
$$
for each irreducible tame $T=T^\la$ with $|\la|\le \rr$. But this is evident, because 
$$
\T_\rr (\th_\rr(\al_n))=P_\rr T(\al_n) P_\rr
$$
and we know that $T(\al_n)\to \Tb(\al)$ strongly. The equality \eqref{eq6.C} is now evident, and it implies in turn the uniqueness claim.

Step 2. The degrees of the elements $b_\rr$ are uniformly bounded. Indeed, this follows from the limit relation \eqref{eq6.B}, because $\deg \th_\rr(\al_n)\le \deg \al_n$, while $\deg\al_n$ remains bounded as $n\to\infty$, by the very definition of convergent sequences.

Step 3.  We claim that 
$$
\th_{r+1,r}(b_{\rr+1})=b_\rr, \quad \rr=1,2,\dots\,.
$$

Indeed, as above, it suffices to prove that  
$$
\T_\rr(\th_{r+1,r}(b_{\rr+1}))=\T_\rr(b_\rr),
$$
where $\T_\rr$ is associated with an irreducible tame representation $T$ such that $H_\rr(T)\ne\{0\}$.

Using \eqref{eq6.C} we obtain
$$
\T_\rr(\th_{r+1,r}(b_{\rr+1}))=P_\rr \T_{\rr+1}(b_{\rr+1}) P_\rr =P_\rr P_{\rr+1} \Tb(\al) P_{\rr+1} P_\rr= P_\rr \Tb(\al)P_\rr =\T_\rr(b_\rr),
$$
as desired. 

Step 4. Suppose $\al\in\B_m$, where $m=0,1,2,\dots$, and let $\Ga_m(\N)\subset \Ga(\N)$ denote the subsemigroup formed by the matrices with the first $m$ diagonal entries equal to $1$. Next, let $T$ be an irreducible tame representation and $\T$ its extension to $\Ga(\N)$. We claim that $\Tb(\al)$ commutes with the operators from $\T(\Ga_m(\N))$. 

Indeed, $\al$ can be represented by  a convergent sequence $(\al_n)$ with $\al_n\in\B_m(n)$ for all $n$ large enough. Because $T(\al_n)\to \Tb(\al)$ in the strong operator topology and $T(\al_n)$ commutes with $T(g)$ for any $g\in S_m(n)$, we see that $\Tb(\al)$ commutes with the operators  from $T(S_m(\infty))$. 

Next, the group $S_m(\infty)$ is dense in $\Ga_m(\N)$. Therefore, $\Tb(\al)$ commutes with the operators from  $\T(\Ga_m(\N))$, too. 

Step 5. Let again $\al\in\B_m$, where $m=0,1,2,\dots$\,. Then $b_\rr\in\B_m(\rr)$ for any $\rr>m$. 

Indeed,  we have to show that 
$$
\ga b_\rr= b_\rr \ga, \quad \ga\in\Ga_m(\rr).
$$
Again, by using Proposition \ref{prop4.A}, we reduce this to the commutation relation 
$$
\T_\rr(\ga)\T_\rr(b_\rr)=\T_\rr(b_\rr)\T_\rr(\ga), \quad \ga\in\Ga_m(\rr).
$$
Next, by virtue of \eqref{eq6.C} we can further reduce it to
$$
\T_\rr(\ga)P_\rr \Tb(\al) P_\rr=P_\rr \Tb(\al) P_\rr\T_\rr(\ga), \quad \ga\in\Ga_m(\rr).
$$

Pick an arbitrary $\wt\ga\in \Ga_m(\N)$ such that $\th_\rr(\wt \ga)=\ga$. Then the above relation can be rewritten as
$$
P_\rr\T(\wt\ga)P_\rr \Tb(\al) P_\rr=P_\rr \Tb(\al) P_\rr\T(\wt\ga) P_\rr. 
$$

Observe now that $P_\rr=\T(\eps_{\N\setminus\N_\rr})$ and $\eps_{\N\setminus\N_\rr}\in\Ga_\rr(\N)\subset \Ga_m(\N)$.  Therefore, by virtue of step 4, the operators $P_\rr$ and $\Tb(\al)$ commute. Therefore, the last equality is equivalent to 
$$
P_\rr\T(\wt\ga)\Tb(\al) P_\rr=P_\rr \Tb(\al)\T(\wt\ga) P_\rr. 
$$

Finally, again by virtue of step 4, the operators $\T(\wt\ga)$ and $\Tb(\al)$ also commute, because $\wt\ga\in\Ga_m(\N)$, and we are done.  

Step 6. The above reasonings show that the sequence $(b_\rr: \rr>m)$ gives rise to an element $b\in\A_m$. Furthermore,  the key relation \eqref{eq6.C} implies that $\Tb(\al)=\wt\T(b)$ for any irreducible tame $T$. Thus, for  any $m=0,1,2,\dots$ we obtain a map $\al\mapsto b$ from $\B_m$ to $\A_m$, which we denote by $\F_m$. It possesses the desired property $\iota_\A\circ \F=\iota_\B$. Since both $\iota_\B$ and $\iota_\A$ are injective algebra homomorphisms, the same holds for $\F_m$. It remains to show that $\F_m$ is surjective. This will be done in the next two steps.

Step 7. The surjectivity of $\F_m$ means that for each $b\in\A_m$ there exists a convergent sequence $\al=(\al_n)$ such that $\al_n\in\B_m(n)$ for all $n$ large enough and $\Tb(\al)=\wt\T(b)$ for any irreducible tame $T$. 

It suffices to prove this for a set of generators of the algebra $\A_m$. By Corollary \ref{cor5.A}, $\A_m$ is generated by the chain \eqref{eq5.B} of subalgebras, the permutations from $S(m)$, and the elements $\eps_1,\dots,\eps_m$. 

If $b$ is a permutation  $g\in S(m)$, the claim is trivial: one can take $\al_n\equiv g$.

Assume now that $b=\eps_i$ for some $i=1,\dots,m$, and let us show that one can take 
\begin{equation}\label{eq6.E}
\al_n=\frac1{n-m}\sum_{j=m+1}^{n}(i,j).
\end{equation}
It is clear that $\al_n\in\B_m(n)$ and $\deg(\al_n)$ does not grow (it is equal to $2$). Next, we have to check that $T^\la(\al_n)$ strongly converges to $\T^\la(\eps_i)$ for any $\la\in\Y$. Because $\Vert T^\la(\al_n)\Vert\le1$, it suffices to check the convergence on a total set of vectors in $H(T^\la)$. Let us use the explicit description of $\T^\la$ given in \eqref{eq4.C}. The Hilbert space $H(T^\la)$ can be decomposed into an orthogonal direct sum of subspaces, each of which has dimension $\dim\pi^\la$:
\begin{equation}\label{eq6.D}
H(T^\la)=\bigoplus_{\om\in\Om(\ell,\infty)} H(\om),
\end{equation}
where $H(\om)\subset H(T^\la)$ denotes the subspace of vector-valued functions supported by the single matrix $\om$. From \eqref{eq4.C} it is seen that the operator $\T^\la(\eps_i)$ is the projection onto the direct sum of those components $H(\om)$, for which the $i$th column in $\om$ consists entirely of $0$'s. Further, for each transposition $(i,j)$ entering \eqref{eq6.E}, the action of the corresponding operator $T^\la((i,j))$ consists in a certain permutation of the components of the orthogonal decomposition \eqref{eq6.D}. From this description one obtains that 
$$
T^\la(\al_n)\big|_{H(\om)}\to \T^\la(\eps_i)\big|_{H(\om)}, \quad \forall \om\in\Om(\ell,\infty).
$$
Indeed, consider separately two cases depending on whether the $i$th column of $\om$ is null or not. In the former case, the transformed matrix $\om(i,j)$ equals $\om$ for all $j$ large enough, which implies that $T^\la(\al_n)f\to f$ for any $f\in H(\om)$. In the latter case, the subspaces $H(\om(i,j))$ are pairwise orthogonal for all $j=m+1,\dots,n$, which implies that $\Vert T^\la(\al_n)f\Vert\to 0$ for any $f\in H(\om)$. This implies the desired strong convergence. 

Step 8. Let us turn now to the subalgebras $\xi^i(\A_0)$ from the chain \eqref{eq5.B}. 

If $b\in\A_0$ then we have
$\wt\T^\la(b)=\wh b(\la)1$, where $\wh b(\la)$ coincides with the restriction to $\Y$ of a shifted symmetric function (see section \ref{sect5.5} and especially Corollary \ref{cor5.B}). Then the desired result, an approximation of $b$ by a convergent system $\al=(\al_n)\in\B_0$, is afforded  by Theorem \ref{thm3.A}. Note that this gives us a strengthening of Theorem \ref{thm3.A}: an isomorphism between the virtual center $\ZZ=\B_0$ and the algebra $\Sym^*$. 

Finally, consider an element  $b':=\xi^i(b)\in\xi^i(\A_0)$, where $i=1,\dots,m$ and $b\in\A_0$, as before. Then an appropriate approximation for $b'$ is the sequence $(\al'_n):=(\xi^i(\al_{n-i}))$, where $\xi^i$ stands for the $i$th iteration of the `shift' endomorphism $\xi:\C[S(\infty)]\to\C[S(\infty)]$.
This completes the proof.
\end{proof}

\begin{remark}
Note an analogy between Theorem \ref{thm6.A} and \cite[Theorem 2.3.4]{Ols91a}. The tame representations $T^\la$ of $S(\infty)$ are similar to the irreducible tame representations of $O(\infty)$ (denoted in \cite{Ols91a} by $\rho^\R_\la$) and the computation in Lemma \ref{lemma3.A} is similar to the computation in \cite[Lemma 2.3.17]{Ols91a}.
\end{remark}

\section{Tame representations of $G(\infty)$}\label{sect7}

From now on we fix a finite group $G$ and denote by $\Psi$ the set of its  irreducible characters. Given $\psi\in\Psi$, we denote by  $\tau^\psi$ the corresponding irreducible representation of $G$. For the purpose of the present paper we need no concrete information about  the characters and representations of $G$: we treat them as purely abstract objects.

\subsection{Multipartitions and irreducible representations of wreath products  $G\wr S(n)$}\label{sect7.1}

The wreath product group $G\wr S(n)$ is the semidirect product of $S(n)$ and the group $G^n=G\times\dots\times G$ ($n$ times) on which $S(n)$ acts in a natural way, by permutations of the factors (see e.g. \cite{CST14}). We use the alternate notation $G(n):=G\wr S(n)$ and realize $G(n)$  as the group of $n\times n$ strictly monomial matrices with entries in $G\sqcup \{0\}$; here `strictly monomial' means that each row and each column contains precisely one nonzero entry.

Recall a well-known construction of irreducible representations of wreath products (see, e.g., \cite[section 2.6]{CST14}).

Let $\Y(\Psi)$ be the set of maps  $\Psi\to\Y$. Elements of this set will be called \emph{$\Psi$-multipartitions} and denoted by the boldface Greek letters $\bla$ or $\bmu$. We set
$$
\|\bla\|:=\sum_{\psi\in\Psi}|\bla(\psi)|
$$
and 
$$
\Y_n(\Psi):=\{\bla: \Vert\bla\Vert=n\}, \quad n=0,1,2,\dots\,.
$$
Next, we  define the \emph{support} of $\bla$ as follows:
$$
\supp\bla:=\{\psi: \bla(\psi)\ne\varnothing\}\subseteq\Psi. 
$$

For $n>0$, the multipartitions $\bla\in\Y_n(\Psi)$ parametrize the irreducible representations of the group $G(n)$. Namely, the representation $\pi^\bla$ corresponding to a given $\bla\in\Y_n(\Psi)$ is obtained as follows. 

\begin{definition}\label{def7.A}
(i) Assume first that  $\supp\bla$ consists of a single element $\varphi\in\Psi$ and set $\nu:=\bla(\varphi)$, $n:=|\nu|$. Then $\pi^\bla$ is the tensor product of $\pi^\nu$ (the irreducible representation of $S(n)$ indexed by $\nu$) and $(\tau^\varphi)^{\otimes n}$ (which is an irreducible representation of $G^n=G\times\dots\times G$).  In this case we use for the resulting representation of $G(n)$ the alternate notation $\pi^{\varphi,\nu}$. 

(ii)  Assume now that $\supp\bla$ consists of $r>1$ points, say, $\varphi(1),\dots,\varphi(r)$. Set 
$$
\nu(1):=\bla(\varphi(1)), \dots, \nu(r):=\bla(\varphi(r)), \quad n_1:=|\nu(1)|, \dots, n_r:=|\nu(r)|, 
$$
and note that $n:=n_1+\dots+n_r=\|\bla\|$. Then $\pi^\bla$  is the following induced representation:
\begin{equation}\label{eq7.C}
\pi^\bla:=\Ind^{G(n)}_{G(n_1)\times \dots \times G(n_r)} \pi^{\varphi(1),\nu(1)}\otimes\dots\otimes\pi^{\varphi(r),\nu(r)}.
\end{equation}
\end{definition}

This parametrization agrees with the one in \cite[Chapter I, Appendix B]{Mac95}. 

\subsection{Analog of Young's branching rule for wreath products}

The classic Young branching rule \eqref{eq2.D} is extended to wreath products as follows. 

Let  $\bla\in\Y_n(\Psi)$, $\bmu\in\Y_{n-1}(\Psi)$, and $\psi\in \Psi$. We write $\bmu\nearrow_{\psi} \bla$ if 
$\bmu(\psi)\nearrow \bla(\psi)$, which automatically entails that $\bmu(\psi')=\bla(\psi')$ for all $\psi'\ne\psi$. With this notation we have
\begin{equation}\label{eq7.A}
\pi^\bla\big|_{G(1)\times G(n-1)}\sim \bigoplus_{\psi\in \Psi} \bigoplus_{\substack{\bmu\in \Y_{n-1}(\Psi):\\ \bmu \nearrow_{\psi}\bla}} \tau^\psi\otimes \pi^\bmu.
\end{equation}

This fact is well known, and its proof is easily derived from the construction of the representations $\pi^\bla$ as induced representations. 
Indeed, if $\supp\bla$ is a single point $\psi$, then the result follows from the classic branching rule, and the general case is reduced to this particular case by making use of the well-known general claim about restriction of an induced representation to a subgroup (see e.g. \cite[\S7]{Se77}).  

\subsection{The irreducible representations $T^\bla$}\label{sect7.3}

Let $G(\N)$ denote the group of all  strictly monomial matrices of the format $\N\times \N$ and with entries in $G\sqcup\{0\}$. We denote the unity element of $G$ by $1$ and identify $G(n)$ with the subgroup of $G(\N)$ formed by the matrices $g=[g_{ij}]$ satisfying the condition $g_{ij}=\de_{ij}$ for $\max(i,j)>n$. We also set $G(\infty):=\bigcup G(n)$.  When $G$ is reduced to $\{1\}$, we return to the groups $S(n)$, $S(\N)$, and $S(\infty)$. 

Next, in the same way, we generalize the definition of the subgroups $S_m(n)\subset S(n)$ and $S_m(\infty)\subset S(\infty)$. Namely, we introduce the subgroups $G_m(n)\subset G(n)$ and $G_m(\infty)\subset G(\infty)$, which are singled out by the requirement that the matrix entries $g_{ij}$ are equal to $\de_{ij}$ for $1\le i,j\le m$. Obviously, $G_m(\infty):=\bigcup_{m>n} G_m(n)$. 

For a unitary representation $T$ of the group $G(\infty)$, we define the subspaces $H_m(T)$ and $H_\infty(T)$ as in section \ref{sect2.1}; the only change is that $S_m(\infty)$ should be replaced by $G_m(\infty)$. Then we can extend to the group $G(\infty)$ the notion of tame representation (see Definition \ref{def2.A}). 

To each $\bla\in\Y_\ell(\Psi)$ with $\ell\in\Z_{\ge1}$ we assign a unitary representation of $G(\infty)$ by analogy with \eqref{eq2.E}:
\begin{equation}\label{eq7.B}
T^\bla:=\Ind^{G(\infty)}_{G(\ell)G_\ell(\infty)} (\pi^\bla\otimes 1).
\end{equation}
It can be described in the same way as in \eqref{eq2.A}, with the matrix space $\Om(\ell,\infty)$ being replaced by the space $\Om_G(\ell,\infty)$, whose definition is analogous, only the matrix entries now take values in $G\sqcup\{0\}$ instead of $\{0,1\}$.

For the only multipartition forming the singleton $\Y_0(\Psi)$ we assign the trivial representation.

All representations $T^\bla$, $\bla\in\Y(\Psi)$, are tame and irreducible. Moreover,  they exhaust all irreducible tame representations. These assertions are proved in exactly the same way as in the particular case of $S(\infty)$ which corresponds to $G=\{1\}$, see \cite{Ols85}, \cite{MO01}.

\subsection{The spectrum of $T^\bla\big|_{G(n)}$}\label{sect7.4}

We are going to extend the definition of diagrams $\la[n]$ (section \ref{sect2.3}). 

Let $\psi_1\in\Psi$ denote the trivial character of $G$. Given $\bla\in\Y(\Psi)$ and $n\ge \Vert\bla\Vert=\ell$, we denote by $\bla[n]$ the element of $\Y_n(\Psi)$ obtained from $\bla$ by adding a new row of length $n-\ell$ to the Young diagram $\bla(\psi_1)$. If $n$ is large enough (so that $n-\ell$ is greater or equal to the length of the first row in $\bla(\psi_1)$), then the new row appears at the top of the diagram $\bla(\psi_1)$ and grows together with $n$, while the diagrams over $\psi\ne\psi_1$ remain unchanged. 

For $\bla, \bmu \in\Y(\Psi)$, we write $\bmu\subset\bla$ if $\bmu(\psi)\subseteq\bla(\psi)$ for each $\psi\in\Psi$ and $\Vert\mu\Vert<\Vert\bla\Vert$.

\begin{proposition}[cf. Proposition \ref{prop2.A}] \label{prop7.A}
Let $\bla\in\Y_\ell(\Psi)$ and let $n$ be large enough, so that $n-\ell$ is greater or equal to the length of the first row in $\bla(\psi_1)$. 

Then the restriction of\/ $T^\bla$ to the subgroup $G(n)$ contains $\pi^{\bla[n]}$ with multiplicity $1$, while all other irreducible components of\/ $T^\bla\big|_{G(n)}$ are of the form $\pi^{\bmu[n]}$, where $\bmu\subset\bla$.
\end{proposition}

\begin{proof}
For $n>\ell$ we set (cf. \eqref{eq2.H})
\begin{equation*}
X_n(\bla):\{\bnu\in\Y_n(\Psi): \bnu(\psi_1)\succ\bla(\psi_1),\;  \bnu(\psi)=\bla(\psi)\; \text{for all $\psi\ne\psi_1$}\}
\end{equation*}
From the construction of $T^\bla$ if follows that the representation of $G(n)$ in $H_n(T^\bla)$ is equivalent to the induced representation
$$
\Ind^{G(n)}_{G(\ell)G_\ell(n)} (\pi^\bla\otimes1).
$$
On the other hand, from the very construction of the irreducible representations of wreath products $G(n)$ it follows that  
\begin{equation}\label{eq7.D}
\Ind^{G(n)}_{G(\ell)G_\ell(n)} (\pi^\bla\otimes1) \sim \bigoplus_{\bnu\in X_n(\bla)} \pi^{\bnu}.
\end{equation}
Then the proof goes as in Proposition \ref{prop2.A}, with the use of the branching rule \eqref{eq7.A}.
\end{proof}

\subsection{Another realization of representations $T^\bla$}

Building on the branching rule \eqref{eq7.A} we can assign to each $\bla\in\Y(\Psi)$ an irreducible representation $\Pi^\bla$ of the group $G(\infty)$,
$$
\Pi^\bla = \varinjlim \pi^{\bla[n]},
$$
in the same way as it was done for the group $S(\infty)$ (section \ref{sect2.4}).  

\begin{proposition}[cf. Proposition \ref{prop2.B}] \label{prop7.B}
For each $\bla\in \Y(\Psi)$, the representations $T^\bla$ and $\Pi^\bla$ are equivalent.
\end{proposition}

\begin{proof}
The same argument as in Proposition \ref{prop2.B}, with the use of the branching rule \eqref{eq7.A} and Proposition \ref{prop7.A}. 
\end{proof}

\subsection{Analog of Young’s basis}\label{sect7.6}

We denote by  $\xi$ the injective endomorphism of $G(\infty)$ defined in the same way as in the beginning of section \ref{sect2.5}.  Its image is $G_1(\infty)$ and we denote by $\xi^{-1}\colon G_1(\infty) \rightarrow G(\infty)$ its inverse. Its superposition with $T^\bla$ produces a representation of $G_1(\infty)$, which we will denote by $T_1^\bla$.

\begin{proposition}[Branching rule]\label{prop7.C}
For any $\bla\in \Y$ the restriction of $T^\bla$ to the subgroup $G\times G_1(\infty)$ of $G(\infty)$ decomposes into a finite, multiplicity free sum of irreducible representations
$$
T^\bla \big|_{G\times G_1(\infty)}\sim(\tau^{\psi_1}\otimes T_1^\bla) \, \oplus \, \bigoplus_{\psi\in \Psi} \bigoplus_{\substack{\bmu\in \Y_{n-1}(\Psi):\\ \bmu \nearrow_{\psi}\bla}} (\tau^\psi\otimes T_1^\bmu).
$$
\end{proposition}

Recall that $\tau^\psi$ denotes the irreducible representation of $G$ indexed by $\psi\in\Psi$. In particular, $\tau^{\psi_1}$ is the trivial representation.

\begin{proof}
We argue as in Proposition \ref{prop2.C}. 
Since the case $\bla = \varnothing$ is trivial, we assume $\|{\bla}\| = \ell\ge 1$ and we work with the initial definition of $T^\bla$, see \eqref{eq7.B}. 

Let us split $\Om_G(\ell,\infty)$ into two parts, $\Om_G(\ell,\infty)^0$ and $\Om_G(\ell,\infty)'$. Here $ \Om_G(\ell,\infty)^0 $ consists of all matrices $\om$ with the null first column and $\Om_G(\ell,\infty)'$ is the rest. Since the action of $G(\ell)$ on left preserves this decomposition, the space $ H(T^{\bla})$ splits into the direct sum $ H^{0}\oplus H' $, where $ H^{0} $ is the subspace of all functions supported by $ \Om_G(\ell,\infty)^{0} $. This splitting is invariant under the action of $ G\times G_1(\infty)$, and the invariant subspace $ H^{0}$ carries the reprsentation $ \tau^{\psi_1}\otimes T_1^{\bla} $. 

Examine now the representation on the subspace $H'$. Let $ \Om_G(\ell,\infty)'' $ denote the subset of all matrices $\om\in \Om_G(\ell,\infty)' $ with $ \omega_{11} = 1$. Note that the shifts of this subset under the left action of $G(\ell)$ cover the whole set  $ \Om_G(\ell,\infty)' $. 

From the branching rule \eqref{eq7.A} it follows that for any $ \psi\in \Psi $ and any $ \bmu \nearrow_{\psi} \bla $ there is a unique $(G\times G_1(n))$-invariant subspace $H(\pi^{\bmu})\subset H(\pi^{\bla}) $ that carries the representation $ \pi^{\bmu} $. Now let $ H^{\bmu} $ be the subspace of $ H' $ defined by
$$
H^{\bmu}:= \{\phi\in H': \phi(\omega)\in  H(\pi^{\bmu})\quad \forall\omega\in \Om_G(\ell,\infty)''\  \}.
$$
It is invariant under the right action of the subgroup $G\times G_1(\infty) $ and carries a representation equivalent to $\tau^\psi\otimes  T_{1}^{\bmu}$.

Finally, the whole space $ H' $ decomposes into the direct sum of subspaces $ H^{\bmu}$ described above.
\end{proof}

\begin{corollary}\label{cor7.A}
For any $\bla\in \Y$ the restriction of $T^\bla$ to the subgroup $G_1(\infty)$ of $G(\infty)$ decomposes into a finite sum of irreducible representations, as follows
\begin{equation}\label{eq7.E}
\left.{T^\bla}\right|_{G_1(\infty)} \sim T_1^\bla \oplus \bigoplus_{\psi\in \Psi} \bigoplus_{\substack{\bmu\in \Y_{n-1}(\Psi):\\ \bmu \nearrow_{\psi}\bla}} (\dim\psi)\, T_1^\bmu.
\end{equation}
\end{corollary}

This result is analogous to  \eqref{eq2.B}. We use it in Proposition \ref{prop10.D}.

\section{Virtual center of the group algebra $\C[G(\infty)]$}\label{sect8}

\subsection{The algebras $\Sym(\Psi)$ and $\SSym(\Psi)$}

To each $\psi\in\Psi$ we assign a copy of the algebra $\Sym$, which we denote by $\Sym^{(\psi)}$, and we set 
$$
\Sym(\Psi) = \bigotimes_{\psi \in \Psi} \Sym^{(\psi)}. 
$$
Because $\Sym$ is a unital algebra, each $\Sym^{(\psi)}$ may be viewed as a subalgebra of $\Sym(\Psi)$: it is formed by the decomposable tensors with the property that all components, except the $\psi$th one, are equal to $1$. Let $I_\psi : \Sym\to \Sym(\Psi)$ be the natural embedding defined by the identification $\Sym=\Sym^{(\psi)}$. 

The following elements are analogs of the Newton power sums $p_k\in\Sym$: 
$$
p_{k,\psi}:=I_\psi (p_k), \qquad k=1,2,\dots\,.
$$
The algebra $\Sym(\Psi)$ is freely generated, as a commutative unital algebra, by the elements $p_{k,\psi}$, where $k=1,2,\dots$ and $\psi\in\Psi$. 

Likewise, we set
$$
\SSym(\Psi) = \bigotimes_{\psi \in \Psi} \SSym^{(\psi)},
$$
where $\SSym^{(\psi)}$ is a copy of $\Sym^*$ indexed by $\psi$. The algebra $\SSym(\Psi)$ has a natural filtration (determined by the filtration of $\SSym$), and the associated graded algebra is $\Sym(\Psi)$. Let $I^*_\psi: \SSym\to\SSym(\Psi)$ be the natural embedding.

Recall that $\SSym$ can be interpreted as an algebra of functions on $\Y$ (see section \ref{sect3}). Likewise, we can regard $\Sym^*(\Psi)$ as an algebra of functions on $\Y(\Psi)$ by setting 
$$
I^*_\psi(f)(\bla):=f(\bla(\psi)), \qquad f\in\SSym, \quad \psi\in\Psi.
$$

\subsection{Central elements $\zz^{(k,\psi)}_{n}$ and their eigenvalues}\label{sect8.2}

We are going to generalize the definition of the central elements $\zz^{(k)}_n$ from section \ref{sect3.2}. Introduce a notation: 

\begin{itemize} 
\item for $g\in G$ and $i=1,\dots,n$, we denote by $g^{(i)}$ the element of $G(n)$ represented by the diagonal $n\times n$ matrix in which the $i$th diagonal entry is $g$ and all other diagonal entries equal $1$;

\item for $\psi\in\Psi$, we denote by $\bar\psi\in\Psi$ the conjugate character, so that $\bar\psi(g)=\overline{\psi(g)}$ for $g\in G$.
\end{itemize}

\begin{definition}\label{def8.A} 
To each pair $(k,\psi)$, where $1\le k\le n$ and $\psi\in\Psi$, we assign an element of the group algebra $\C[G(n)]$ by setting
\begin{equation}\label{eq8.A}
\zz^{(k, \psi)}_{n}:=\left(\frac{\dim\psi}{|G|}\right)^k\sum \bar\psi(g_k\dots g_1)g_1^{(i_1)}\dots g_k^{(i_k)}(i_1,\dots,i_k),
\end{equation}
where the sum is taken over all ordered $k$-tuples of pairwise distinct indices $i_1,\dots,i_k$ in $\N_n$ and arbitrary ordered $k$-tuples of elements $g_1,\dots,g_k$ in $G$. 
\end{definition}

Note that for $G=\{1\}$, the definition reduces to  \eqref{eq3.F1}.

Let $\ZZ\C[G(n)]$ denote the center of $\C[G(n)]$. 

\begin{lemma}\label{lemma8.A}
The elements $\zz^{(k, \psi)}_{n}$ lie in \/ $\ZZ\C[G(n)]$.
\end{lemma}

\begin{proof}
The group $G(n)$ is generated by its subgroups $S(n)$ and $G^n$. Therefore, it suffices to check that 
$$
s\, \zz^{(k, \psi)}_{n} s^{-1}= \zz^{(k, \psi)}_{n}, \quad \forall\, s\in S(n); \qquad 
h\, \zz^{(k, \psi)}_{n} h^{-1}= \zz^{(k, \psi)}_{n}, \quad \forall\, h\in G^n. 
$$
The first equality is evident, so we turn to the second one. In \eqref{eq8.A}, we have a double sum, over $(i_1,\dots,i_k)$ and over $(g_1,\dots,g_k)$. It suffices to check that for $(i_1,\dots,i_k)$ fixed, the conjugation by $h$ leaves invariant the sum over $(g_1,\dots,g_k)$. 

To simplify the notation, let us examine the particular case when $k=n$ and the cycle $(i_1,\dots,i_k)$ is $(1,\dots,n)$; in the general case, the computation is the same. 

Write $h=h_1^{(1)}\dots h_n^{(n)}$, where $h_1,\dots, h_n\in G$ . Then the conjugation 
$$
g_1^{(1)}\dots g_n^{(n)} (1,\dots,n)\, \longrightarrow\, h\, g_1^{(1)}\dots g_n^{(n)} (1,\dots,n)\, h^{-1}
$$
boils down to the transformation 
$$
g_1\to h_1 g_1 h_n^{-1}, \; g_2 \to h_2 g_2 h_1^{-1},\; \dots, \;g_n\to h_n g_n h_{n-1}^{-1}.
$$
Under this transformation, the product $g_n\dots g_1$ is conjugated by $h_n$, which does not affect the quantity $\bar\psi(g_n\dots g_1)$, because $\psi$ is a class function. This concludes the proof.
\end{proof}

For a central element $c\in\ZZ\C[G(n)]$, we denote by $\wh c(\bla)$ its eigenvalue in the irreducible representation $\pi^\bla$, $\bla\in\Y_n(\Psi)$. 

Next, recall the definition of the elements $p^\#_\rho\in\Sym^*$, see \eqref{eq3.J} and \eqref{eq3.K}. In what follows we are dealing with a particular case of these elements, $p^\#_k:=p^\#_{(k)}$.

For $k\in\Z_{\ge1}$ and $\psi\in\Psi$, we set
\begin{equation}\label{eq8.B}
p^\#_{k,\psi}:=I^*_\psi(p^\#_k)\in\SSym(\Psi), \quad \psi\in\Psi.
\end{equation}

Since the elements $\zz_n^{\;(k,\psi)}$ are central (Lemma  \ref{lemma8.A}), the quantities $\wh\zz_n^{\;(k,\psi)}(\bla)$ are well defined.

\begin{proposition}[cf. Proposition \ref{prop3.A}] \label{prop8.A}
Let, as above, $1\le k\le n$ and $\psi\in\Psi$.  
For any $\bla\in\Y_n(\Psi)$ one has
\begin{equation}\label{eq8.J}
\wh\zz_n^{\;(k,\psi)}(\bla) = p_{k,\psi}^\#(\bla)=p_k^\#(\bla(\psi)). 
\end{equation}
\end{proposition}

\begin{proof}

Step 1. We start with the evident formula
\begin{equation}\label{eq8.C}
\wh\zz_n^{\;(k,\psi)}(\bla)=\dfrac{\tr\pi^\bla(\zz_n^{(k,\psi)})}{\dim \pi^\bla}.
\end{equation}

In \eqref{eq8.A}, the conjugacy class of the element $g_1^{(i_1)}\dots g_k^{(i_k)}(i_1,\dots,i_k)\in G(n)$ depends only on the product $g:=g_k\dots g_1$ (see \cite{Mac95}, Appendix B, section 3, for a description of conjugacy classes in wreath products). Note also that the number of $k$-tuples $(i_1,\dots,i_k)$ equals $n^{\da k}$. It follows that  we can rewrite \eqref{eq8.C} as
\begin{equation}\label{eq8.D}
\wh\zz_n^{\;(k,\psi)}(\bla)=\frac{n^{\da k}(\dim\psi)^k}{\dim\pi^\bla}\cdot\frac1{|G|}\tr \pi^\bla\bigg(\sum_{g\in G}\bar\psi(g) g^{(1)}(1,\dots,k)\bigg).
\end{equation}
  
Step 2. Examine first the case when $\supp\bla$ consists of a single element, say, $\varphi\in\Psi$. Then $\pi^\bla$ takes the form $\pi^{\varphi,\nu}$ for some $\nu\in\Y_n$ (see Definition \ref{def7.A}, item (i)). Set $d:=\dim \tau^\varphi$ and pick a basis $e_1,\dots,e_d$ in the representation space $H(\tau^\varphi)$. Next, let $V=H(\pi^\nu)$ stand for the space of the representation $\pi^\nu$ of the $S(n)$. The whole space $H(\pi^\bla)$ can be written as 
\begin{equation*}
\bigoplus_{j_1,\dots,j_n=1}^d V\otimes e_{j_1}\otimes\dots\otimes e_{j_n}.
\end{equation*}
Given $v\in V$, the result of the action of the element $\bar\psi(g) g^{(1)}(1,\dots,k)$ on the vector 
\begin{equation}\label{eq8.E}
v\otimes  e_{j_1}\otimes\dots\otimes e_{j_n}
\end{equation}
equals
\begin{equation}\label{eq8.F}
\bar\psi(g) (\pi^\nu((1,\dots,k))v)\otimes \tau^\varphi(g) e_{j_k}\otimes e_{j_1}\otimes\dots\otimes e_{j_{k-1}}\otimes e_{j_{k+1}}\otimes\dots\otimes e_{j_n}.
\end{equation}
It follows that the scalar product of the vectors \eqref{eq8.F} and \eqref{eq8.E} vanishes unless 
$$
j_1=j_2,\quad  j_2=j_3,\quad  \dots,\quad j_{k-1}=j_k,
$$
that is, $j_1=\dots=j_k$. 

Using this we can compute the trace in \eqref{eq8.D}:
\begin{multline}\label{eq8.G}
\frac1{|G|}\tr \pi^\bla\bigg(\sum_{g\in G}\bar\psi(g) g^{(1)}(1,\dots,k)\bigg)=\chi^\nu_{(1,\dots,k)}(\dim \varphi)^{n-k}\frac1{|G|}\sum_{g\in G}\bar\psi(g)\varphi(g)\\
=\chi^\nu_{(1,\dots,k)}(\dim \varphi)^{n-k}\cdot\de_{\psi,\varphi},
\end{multline}
by virtue of the orthogonality relation for irreducible characters.

To obtain the right-hand side of \eqref{eq8.D} we have to multiply the final expression in \eqref{eq8.G} by 
$$
\frac{n^{\da k}(\dim\psi)^k}{\dim \pi^\bla}=\frac{n^{\da k}(\dim\psi)^k}{\dim\nu (\dim \varphi)^n}.
$$
After cancellations we obtain 
$$
\frac{n^{\da k} \chi^\nu_{(1,\dots,k)}}{\dim\nu}\cdot\de_{\psi,\varphi}=p^\#_k(\nu)\cdot \de_{\psi,\varphi}.
$$
In the particular case under consideration, when $\supp\bla$ is the singleton $\{\varphi\}$ and $\bla(\varphi)=\nu$, the resulting expression coincides with $p^\#_{k,\psi}(\bla)$, as desired.

Step 3. Let us turn now to the case when $\supp \bla$ is not a singleton. Recall the classic formula for induced characters: if $\mathcal G$ is a finite group, $\mathcal H$ its subgroup, $\chi$ a character of $\mathcal H$, and $\Ind \chi$ the induced character of $\mathcal G$, then the value of $\Ind\chi$ at a given element $x\in\mathcal G$ is given by
\begin{equation}\label{eq8.H}
\Ind\chi(x)=\sum_{y\in[\mathcal H\setminus\mathcal G]}\chi^0(yxy^{-1}),
\end{equation}
where $[\mathcal H\setminus\mathcal G]$ is an arbitrary system of representatives for cosets modulo $\mathcal H$ and $\chi^0$ stands for the continuation of $\chi$ by zero outside the subgroup $\mathcal H$. 

According to Definition \ref{def7.A}, item (ii), $\pi^\bla$ is an induced representation. To compute the trace in \eqref{eq8.D} we apply \eqref{eq8.H} to our concrete situation: we set $\mathcal G:=G(n)$, 
$\mathcal H:=G(n_1)\times\dots\times G(n_r)$, and take for $\chi$ the character of $\pi^{\varphi(1),\nu(1)}\otimes\dots\otimes\pi^{\varphi(r),\nu(r)}$,  see \eqref{eq7.C}. 

Next, observe that the condition that a given $k$-cycle $c\in S(n)$ lies in a given Young subgroup just means that $c$ fits in one of the blocs of that subgroup. Likewise, for an element $x\in G(n)$ of the form $g^{(1)}(1,\dots,k)$, the condition that its conjugate $yxy^{-1}$ belongs to the subgroup $G(n_1)\times\dots\times G(n_r)$ just means that $yxy^{-1}$ fits into one of the blocs $G(n_i)$. 

Combining this fact with the formula \eqref{eq8.H} and our previous argument we see that the trace in \eqref{eq8.D} vanishes unless the following condition holds: $\psi=\varphi(i)$ for some index $i\in\{1,\dots,r\}$ and, moreover, $k\le n_i$. On the other hand, if this is not the case, then $p^\#_{k,\psi}(\bla)$ vanishes, too, by the very definition of $p^\#_{k,\psi}$. 

It remains to examine the case when the above condition holds true. With no loss of generality we may assume that $i=1$. Thus, $\psi=\varphi(1)$ and $k\le n_1$. 

Then, by the definition of $p^\#_{k,\psi}$, we have
$p^\#_{k,\psi}(\bla)=p^\#_k(\nu(1))$. Therefore, we have to check that the formula \eqref{eq8.D} also gives $p^\#_k(\nu(1))$. To do this, we apply the formula \eqref{eq8.H} to $x:=g^{(1)}(1,\dots,k)$. Because 
$$
G(n_1)\times \dots\times G(n_r)\setminus G(n)=S(n_1)\times \dots\times S(n_r)\setminus S(n),
$$
we may assume that $y$ ranges over $[S(n_1)\times \dots\times S(n_r)\setminus S(n)]$. 

Note that $x$ is already contained in the subgroup
$$
G(n_1)\times\{1\}\times\dots\times\{1\}\subset G(n_1)\times G(n_2)\times\dots\times G(n_r).
$$ 
There are precisely 
$$
\frac{n_1^{\da k}(n-k)!}{n_1!\dots n_r!}
$$
permutations $y\in [S(n_1)\times \dots\times S(n_r)\setminus S(n)]$ such that the corresponding element $y x y^{-1}$ lies in the same subgroup, and then 
$$
\chi^0(y x y^{-1})=\chi(y x y^{-1})=\chi(x)=\tr\pi^{\varphi(1),\nu(1)}\big(g^{(1)}(1,\dots,k)\big)\cdot \prod_{i=2}^r \dim \pi^{\varphi(i),\nu(i)}.
$$

It follows that
\begin{multline}\label{eq8.I}
\frac1{|G|}\tr \pi^\bla\bigg(\sum_{g\in G}\bar\psi(g)g^{(1)}(1,\dots,k)\bigg)\\
=\frac1{|G|}\tr\pi^{\varphi(1),\nu(1)}\bigg(\sum_{g\in G}g^{(1)}(1,\dots,k)\bigg)\cdot \frac{n_1^{\da k}(n-k)!}{n_1!\dots n_r!}
\cdot \prod_{i=2}^r \dim \pi^{\varphi(i),\nu(i)}.
\end{multline}

Next, by the result of Step 2 (formula \eqref{eq8.G}),
$$
\frac1{|G|}\tr \pi^{\varphi(1),\nu(1)}\bigg(\sum_{g\in G}g^{(1)}(1,\dots,k)\bigg)=\chi^{\nu(1)}_{(1,\dots,k)}\cdot (\dim \varphi(1))^{n_1-k}=\chi^{\nu(1)}_{(1,\dots,k)}\cdot(\dim \psi)^{n_1-k}.
$$
Now we substitute this into \eqref{eq8.I} and then multiply by the resulting expression by 
$$
\frac{n^{\da k}(\dim\psi)^k}{\dim\pi^\bla},
$$
as it is prescribed by \eqref{eq8.D}. Finally, we also use the fact that
\begin{multline*}
\dim\pi^\bla=\frac{n!}{n_1!\dots n_r!}\prod_{i=1}^r\dim \pi^{\varphi(i),\nu(i)}
= \frac{n!}{n_1!\dots n_r!}\dim\pi^{\psi,\nu(1)}\cdot \prod_{i=2}^r\dim \pi^{\varphi(i),\nu(i)}\\
= \frac{n!}{n_1!\dots n_r!}\dim\nu(1)\cdot(\dim\psi)^{n_1}\cdot \prod_{i=2}^r\dim \pi^{\varphi(i),\nu(i)}.
\end{multline*}
After cancellations we obtain the final expression 
$$
\frac{n_1^{\da k} \chi^{\nu(1)}_{(1,\dots,k)}}{\dim\nu(1)}=p^\#_k(\nu(1)),
$$
as required.  This completes the proof. 
\end{proof}

\subsection{Filtration in $\C[G(\infty)]$}

The \emph{degree} of an element  $g\in G(\infty)$ is defined exactly as in section \ref{sect3.3}: $\deg g$ is the number of diagonal entries $g_{ii}$ distinct from $1$. The formula \eqref{eq3.H} still holds, and we extend the notion of degree to the group algebra $\C[G(\infty)]$ as in \eqref{eq3.H1}. In this way we obtain a filtration in $\C[G(\infty)]$. 

The next assertion is a counterpart of Corollary \ref{cor3.A}.

\begin{corollary}\label{cor8.A}
For any function $f^*\in \SSym(G)$ there exists a sequence of elements $c_n(f^{\ast})\in \ZZ\C[G(n)]$, $n=1,2,3,\dots$,  such that for any $n$ one has 
$$ 
\deg c_n(f^*)\le\deg f^*
$$ 
and 
$$
\wh{c_n(f^*)}=f^*(\bla), \quad \bla\in\Y_n(\Psi). 
$$
\end{corollary}

\begin{proof}
The same argument as in Corollary \ref{cor3.A}, with reference to Proposition \ref{prop8.A} instead of Proposition \ref{prop3.A}. A minor difference is that for $G\ne \{1\}$,  the case $k=1$ presents no exception: $\deg \zz^{(1, \psi)}_{n}$ equals $1$, not $0$.  
\end{proof}

\subsection{Construction of virtual central elements}

We introduce the \emph{virtual center} $\ZZ(G)$ of the group algebra $\C[G(\infty)]$ following Definitions \ref{def3.A} and \ref{def3.B}, by making use of representations $T^\bla$, $\bla\in\Y(\Psi)$, instead of representations $T^\la$, $\la\in\Y$.

\begin{theorem}[cf. Theorem \ref{thm3.A}]\label{thm8.A}
We have $\ZZ(G)\supseteq\Sym^*(\Psi)$, with the understanding that elements of both algebras are treated as functions on $\Y(\Psi)$. 
\end{theorem}

Actually, $\ZZ(G)=\Sym^*(\Psi)$: this is a particular case of the last assertion of Theorem \ref{thm11.A} corresponding to $m=0$. 

\begin{proof}
We apply the same argument as in Theorem \ref{thm3.A}. That theorem relies on Lemma \ref{lemma3.A} (we apply it again) and also on Proposition \ref{prop2.A}, Proposition \ref{prop2.B}, and Corollary \ref{cor3.A} (we replace them by their counterparts Proposition \ref{prop7.A}, Proposition \ref{prop7.B}, and Corollary \ref{cor8.A}, respectively). \end{proof}

\section{The semigroup method for wreath product groups}\label{sect9}

\subsection{The inverse semigroups $\Ga(\N,G)$ and $\Ga(n,G)$}

Let $\Ga(\N,G)$ be the set of $\N \times\N$ matrices with the entries in $G\sqcup\{0\}$ and such that in each row and each column there is at most one nonzero entry. It is a semigroup under matrix multiplication. We endow it with the involution $\ga\mapsto\ga^*$, where 
$$
(\ga^*)_{ij}:=\begin{cases} (\ga_{ji})^{-1}, & \ga_{ji}\ne0, \\ 0, & \ga_{ji}=0. \end{cases}
$$ 
In the case of the trivial group $G=\{1\}$ this definition reduces to that of the semigroup $\Ga(\N)$. Like the latter semigroup, $\Ga(\N,G)$ is an inverse semigroup. 

For $n\in\Z_{\ge1}$, we define $\Ga(n,G)$ in the similar way, only the matrices are now of the format $n\times n$. We also regard $\Ga(n,G)$ as a subsemigroup of $\Ga(\N,G)$: it consists of the matrices $\ga$ subject to the condition $\ga_{ij}=\de_{ij}$ whenever $i>n$ or $j>n$. 

We define the topology on $\Ga(\N,G)$ in the same way as for $\Ga(\N)$: it is induced from the product topology on the set $(G\sqcup\{0\})^{\N\times\N}$. Again, this gives us a compact semitopological semigroup. 

The group $G(\N)$ is contained in $\Ga(\N,G)$ and coincides with the subset of invertible elements. It is a dense subset.  Moreover, the subgroup $G(\infty)\subset G(\N)$ is also dense in $\Ga(\N,G)$.

\subsection{The semigroup method}\label{sect9.2}

All facts of section \ref{sect4} extend smoothly to the more general case of wreath products and related semigroups, with the same proofs. So we only briefly list the results. 

1. The irreducible representations of $\Ga(n,G)$ are indexed by $\Psi$-multipartions $\bla$ with $\Vert\bla\Vert\le n$; we denote them by $\T^\bla_n$. The representation $\T^\varnothing$ is the trivial one-dimensional reprsentation sending all elements to $1$. The representations $\T^\bla_n$ with $\Vert\bla\Vert=\ell\ne0$ are given by the formula \eqref{eq4.B}; the only change is that $\Om(\ell,n)$ is replaced by $\Om_G(\ell,n)$ (the set of $\ell\times n$ matrices with the entries in $G\sqcup\{0\}$, containing exactly one non-null element in each row and at most one non-null element in each column).

2. Because $\Ga(n,G)$ is an inverse semigroup, its semigroup algebra $\C[\Ga(n,G)]$ is a $C^*$-algebra. Another way to see this is to observe that the group $G$ can be embedded into the symmetric group $S(|G|)$, which implies that  we can realize $\Ga(n,G)$ as an involutive subsemigroup of the symmetric inverse semigroup $\Ga(n|G|)$, and then we can apply Proposition \ref{prop4.A}.

3. The truncation map $\th_n: \Ga(\N)\to \Ga(n,G)$ is defined in the same way. It is surjective; even more, its restriction to $G(2n)\subset G(\infty)\subset \Ga(\N)$ is already surjective.  

4. The representations $T^\bla$ of the group $G(\infty)$ are canonically extended to representations $\T^\bla$ of the semigroup $\Ga(\N)$, and the latter are linked to the representations $\T^\bla_n$ of the semigroups $\Ga(n,G)$ --- exactly as described in Proposition \ref{prop4.C} and Remark \ref{rem4.A}. 

5. The canonical embeddings $\Ga(n,G)\hookrightarrow \Ga(n+1,G)$ are defined as in section \ref{sect4.8}, and then we set $\Ga(\infty,G):=\varinjlim \Ga(n,G)$.

\section{The centralizer construction for wreath product groups}\label{sect10}

\subsection{The algebras $\A_m(G)$ and $\A(G)$}

Here we extend the definitions of section \ref{sect5.1} 

Given a matrix $\ga\in\Ga(n, G)$, we define its \emph{degree}, $\deg (\ga)$, as the number of diagonal entries distinct from $1$. For $G=\{1\}$ this agrees with the previous definition.

We denote by $\A(n,G)$ the semigroup algebra $\mathbb{C}[\Ga(n,G)]$ and endow it with the ascending filtration induced by $\deg(\,\cdot\,)$. That is, the $k$th term of the filtration is spanned by the semigroup elements with degree at most $k$. Note that (\cite[Lemma~2.10]{Wan10})
$$
\deg(\ga_1\ga_2)\le\deg(\ga_1)+\deg(\ga_2), \qquad \ga_1,\ga_2\in\A(n,G),
$$
so $\A(n,G)$ becomes a filtered algebra.

The projection
$$
\theta_{n,n-1}: \Gamma(n,G)\rightarrow \Gamma(n-1,G), \quad \ga\mapsto\theta_{n,n-1}(\ga),
$$
consists in striking from $\ga$ its last row and column. We extend $\theta_{n,n-1}$  to a linear map $\theta_{n,n-1}: \A(n,G)\rightarrow\A(n-1,G)$. We still have (\cite[Lemma~2.7]{Wan10})
$$
\deg(\theta_{n,n-1}(\gamma))\le \deg(\gamma).
$$

\begin{definition}[cf. Definition \ref{def5.A}]\label{def10.A}
Let $m=0,\dots,n$.

(i) We denote by $\Ga_m(n,G)\subseteq \Ga(n,G)$ the subsemigroup of elements with first $m$ diagonal entries equal to $1$. Note that  $\Ga_m(n,G)$ is isomorphic to $\Ga(n-m,G)$ (by definition, $\Ga(0,G)$ consists of a single element, which is automatically the unit). 

(ii) We denote by $\A_m(n,G)\subseteq\A(n,G)$ the centralizer of $\Ga_m(n,G)$ in $\A(n,G)$. Note that $\A_0(n,G)$ is the center of $\A(n,G)$ and $\A_n(n,G)=\A(n,G)$.
\end{definition}

The restriction of $\theta_{n,n-1}$ to $\A_{n-1}(n,G)\subset \A(n,G)$ defines a unital algebra morphism
$\A_{n-1}(n,G)\rightarrow \A(n-1,G)$; moreover,
$$
\theta_{n,n-1}(\A_m(n,G))\subset \A_m(n-1,G)\quad \text{for}\quad 0\le m \le n-1.
$$
The proof is the same as in \cite[Proposition~3.3]{MO01}, see \cite[Lemma 2.7]{Wan10}.

Thus,  for any $m\ge 0$ we have a projective system of filtered algebras
$$
\A_m(m,G)\leftarrow \A_m(m+1,G) \leftarrow \cdots \leftarrow \A_m(n,G) \leftarrow \cdots\,.
$$

\begin{definition}[The centralizer construction, cf. Definition \ref{def5.B}]\label{def10.B} Let $m=0,1,2,\dots$\,. 

(i) We set
$$
\A_m(G):= \varprojlim\A_m(n,G), \quad n\to\infty,
$$
where the projective limit is taken in the category of filtered algebras.  
In more detail, each element $a\in\A_m(G)$ is a sequence $(a_n: n\ge m)$ such that
$$
a_n\in \A_m(n,G),\quad a_{n-1} = \theta_{n,n-1}(a_n), \quad \deg(a) := \sup \deg(a_n)<\infty.
$$

(ii) Because $\A_m(G)\subset \A_{m+1}(G)$,  the following definition makes sense 
$$
\A(G):= \varinjlim \A_m(G), \quad m\to \infty.
$$
\end{definition}	

There is a natural embedding $\C[G(\infty)]\hookrightarrow\A(G)$ whose image consists of stable sequences $a=(a_n)$.
As before, `stable' means that $a_n$ does not depend  on $n$ as $n$ gets large enough.
In particular, the image of an element $g\in G(\infty)$ is represented by the sequence $(a_n\equiv g)\in\A_m(G)$, where $m$ is chosen so large that $g\in G(m)$.

\subsection{Structure of the algebra $\A_0(G)$}

As was pointed out in section \ref{sect5.2}, the algebra $\A_0$ is the center of $\A$. More generally (\cite[Proposition 2.13]{Wan10}), the algebra $\A_0(G)$ is the center of $\A(G)$. 

We are going to generalize the definition of the elements $\De^{(k)}\in\A_0$ from section \ref{sect5.2}.

\begin{definition}[cf. \eqref{eq5.H} and \eqref{eq8.A}]\label{def10.C}
Let $\varphi$ be a class function on $G$. 

(i) For $1\le k\le n$ we set 
\begin{equation}\label{eq10.A}
\De_n(k,\varphi):=\sum \varphi(g_k\dots g_1)g_1^{(i_1)}\dots g_k^{(i_k)}(i_1,\dots,i_k)\beps_{i_1}\dots\beps_{i_k},
\end{equation}
where the sum is taken over all ordered $k$-tuples of pairwise distinct indices $i_1,\dots,i_k$ in $\N_n$ and aribitrary ordered $k$-tuples of elements $g_1,\dots,g_k$ in $G$. 

(ii) We agree that $\De_n(k,\varphi)=0$ for $n<k$. 
\end{definition}

In the case when $\varphi$ is the indicator of a conjugacy class in $G$, the corresponding elements \eqref{eq10.A} are proportional to the elements from \cite[(3.1)]{Wan10}. 

\begin{lemma}\label{lemma10.B}
The elements \eqref{eq10.A} belong to $\A_0(n,G)$, the center of\/ $\C[\Ga(n,G)]$.  
\end{lemma}

\begin{proof}
The same argument as in \cite[Lemma 3.7]{Wan10}. The key observation is that 
$$
g_1^{(i_1)}\dots g_k^{(i_k)}(i_1,\dots,i_k)\beps_{i_1}\dots\beps_{i_k}=\beps_{i_1}\dots\beps_{i_k} g_1^{(i_1)}\dots g_k^{(i_k)}(i_1,\dots,i_k).
$$
\end{proof}

\begin{lemma}\label{lemma10.C}
Fix $\varphi$ and $k$, and let $n$ vary. The corresponding elements \eqref{eq10.A} are consistent with the projections $\th_{n,n-1}$ and hence they determine an element of $\A_0(G)$, which will be denoted by $\De(k,\varphi)$.
\end{lemma}

\begin{proof}
The same argument as in \cite[Proposition 3.8]{Wan10}.
\end{proof}

\begin{proposition}\label{prop10.A}
Pick an arbitrary basis $\Phi$ in the space of class functions on $G$.  The elements $\De(k,\varphi)$ with $k=1,2,3\dots$ and $\varphi\in\Phi$ 
are algebraically independent generators of the filtered algebra $\A_0(G)$. 
\end{proposition}

More precisely, this means that for any given $d=1,2,\dots$, each element of $\A_0(G)$ of degree $\le d$ can be written, in a unique way, as a polynomial in the generators, of total degree $\le d$, with the understanding that $\deg \De(k,\varphi)=k$.
 
\begin{proof}
In the case when $\Phi$ consists of the indicators of conjugacy classes in $G$ this follows from \cite[Corollary 3.17]{Wan10}. Because $\De(k,\varphi)$ depends on $\varphi$ linearly, one can then pass to any other basis $\Phi$.  
\end{proof}

\subsection{Structure of the algebras $\A_m(G)$, $m\ge 1 $}

Following \cite[(4.11)]{Wan10} we extend the definition \eqref{eq5.G} as follows:
$$
u_{i\mid n}(G):= \sum_{j = i+1}^n \sum_{g\in G} g^{(i)}(g^{-1})^{(j)}(i,j) \beps_{i}\beps_{j}=\sum_{j = i+1}^n \sum_{g\in G} \beps_{i}\beps_{j}g^{(i)}(g^{-1})^{(j)}(i,j), \quad 1\le i<n.
$$
These are elements of the semigroup algebra $\A(n,G)$, of degree $2$. We also set $u_{i\mid n}(G):=0$ for $n\le i$. 

It is easily checked that 
$$
u_{1\mid n}(G), \dots, u_{m\mid n}(G) \in\A_m(n), \quad 1\le m<n,
$$
and $ \theta_{n,n-1}(u_{i\mid n}(G)) = u_{i\mid n-1}(G)$.

It follows that for each fixed $i$, the sequence $(u_{i\mid n}(G): n>i)$ gives rise to an element of the algebra $\A(G)$, which we denote by $u_i(G)$. Further, 
$$
u_1(G), \dots, u_m(G)\in \A_m(G), \quad \text{for any $m\ge1$}.
$$

\begin{definition}[cf. Definition \ref{def5.C}]\label{def10.D}
For $m=1,2,3,\dots$, let $\wt\HH_m(G)\subset\A_m(G)$ be the subalgebra generated by $\Ga(m,G)$ and $u_1(G),\dots,u_m(G)$. Equivalently, $\wt\HH_m(G)\subset\A_m(G)$ is generated by $G(m)$, $\eps_1,\dots,\eps_m$, and $u_1(G),\dots,u_m(G)$.
\end{definition}

\begin{lemma}[cf. Lemma \ref{lemma5.A}]\label{lemma10.D}
Fix $m\in\Z_{\ge1}$ and let $h$ range over the subgroup $G^m\subset\A(m,G)$. In the algebra $\A_m(G)$, the following commutation relations hold:
\begin{gather*}
h u_k(G)=u_k(G) h,  \quad 1\le k\le m,\\
s_ku_k(G) = u_{k+1}(G) s_k + \sum_{g\in G}g^{(k)}(g^{-1})^{(k+1)}\beps_k\beps_{k+1}, \quad 1\le k\le m-1,\\
s_k u_l(G)=u_l(G) s_k, \quad l\ne k, k+1, \; 1\le k,l\le m\\
u_k(G) u_l(G)=u_l(G)u_k(G), \quad  1\le k, l\le m, \\
\eps_k u_k(G)=u_k(G)\eps_k=0, \quad \eps_l u_k(G)=u_k(G)\eps_l, \quad 1\le l\ne k\le m.
\end{gather*}
\end{lemma}

\begin{proof}
See \cite[Lemma 4.12]{Wan10}.
\end{proof}

\begin{proposition}[cf. Proposition \ref{prop5.B}]\label{prop10.B}
Fix $m\in\Z_{\ge1}$.

{\rm(i)} The natural morphism $\A_0(G)\otimes\wt\HH_m(G)\to \A_m(G)$ induced by the multiplication map is an algebra isomorphism.  

{\rm(ii)} The  algebra $\wt\HH_m(G)$ is isomorphic to the abstract algebra generated by the semigroup algebra $\A(m,G)=\C[G(m)]$ and the elements $u_1(G),\dots,u_m(G)$, subject to the commutation relations from Lemma \ref{lemma10.D}.
\end{proposition}

\begin{proof}
See \cite[Theorem~4.13]{Wan10}. 
\end{proof}

\subsection{The use of endomorphism $\xi$ of the algebra $\A(G)$}

We introduce an injective endomorphism $\xi: \A(G)\to\A(G)$ with the property that $\xi(\A_m(G))=\A_{m+1}(G)$ for each $m$; its definition is the same as in section \ref{sect5.4}. 

\begin{lemma}[cf. Lemma \ref{lemma5.B}]\label{lemma10.E}
Let $\varphi_1(g)$ denote the class function on $G$ that equals $1$ at the unity and $0$ for all other $g\in G$.  

One has
$$
2u_i(G) = \xi^{i-1}(\De(2,\varphi_1)) - \xi^i(\De(2,\varphi_1)).
$$
\end{lemma}

\begin{proof}
Observe that, by \eqref{eq10.A},
$$
\De_n(2,\varphi_1)=\sum_{1\le i\ne j\le n} \sum_{g\in G} g^{(i)} (g^{-1})^{(j)}(i,j)\beps_i\beps_j.
$$
Then one can apply the same argument as in Lemma \ref{lemma5.B}.
\end{proof}

Combining this with Proposition \ref{prop10.B} we obtain

\begin{corollary}\label{cor10.A}
For any $ m\in\Z_{\ge1}$ the algebra $\A_m(G)$ is generated by its subalgebras
$$
\A_0(G), \; \xi(\A_0(G)), \; \ldots,\; \xi^m(\A_0(G)),
$$
together with $\Ga(m,G)$. 
\end{corollary}

\subsection{Extension of tame representations to the algebra $ \A(G) $}

Any tame representation $T$ of $ G(\infty) $ gives rise to a representation $\wt \T$ of the algebra $\A(G)$ acting on the dense subspace $H_\infty(T)$. If $T=T^\bla$ for some $\bla\in\Y(\Psi)$, then the corresponding representation $\wt \T^\bla$ is algebraically irreducible, which implies that $\A_0(G)$ acts on $H_\infty(T^\bla)$ by scalar operators:
\begin{equation}\label{eq10.E}
\wt \T^\bla(a)=\wh a(\bla)\cdot 1, \quad a\in\A_0(G),
\end{equation}
cf. \eqref{eq5.E}. 
All this is proved exactly as in the case of $S(\infty)$, see \cite[Propositions 3.10 and 3.11]{MO01}. 

In particular, the last equality holds for the central elements  $\De(k,\varphi)\in\A_0(G)$, where $\varphi$ is a class function on $G$ (see \eqref{eq10.A}):
\begin{equation*}
\wt \T^\bla(\De(k,\varphi))=\wh\De(k,\varphi)(\bla)\cdot 1. 
\end{equation*}

The next result is a generalization of Proposition \ref{prop5.C}. 

\begin{proposition}\label{prop10.C}
Let $k=1,2,\dots$,  $\psi\in\Psi$, and set
\begin{equation}\label{eq10.C}
\varphi:=\left(\frac{\dim\psi}{|G|}\right)^k \bar\psi.
\end{equation}
For any $\bla\in\Y(\Psi)$,
\begin{equation}\label{eq10.B}
\wh\De(k,\varphi)(\bla)=p^\#_{k,\psi}(\bla),
\end{equation}
where the elements $p^\#_{k,\psi}\in\Sym^*(\Psi)$ are defined in  \eqref{eq8.B}.
\end{proposition}

\begin{proof}
We adapt the argument from \cite[Proposition 4.20]{MO01}. 
Denote $\Vert\bla\Vert=\ell$. This number is the conductor of $T^\bla$, which implies the following consequences:

\begin{itemize}

\item[(1)] the subspace $H_\ell(T^\bla)$ is nonzero;
\item[(2)] all elements $\eps_i$ with $1\le i\le\ell$ act on it by $0$'s;
\item[(3)] the representation of $G(\ell)\subset G(\infty)$ on $H_\ell(T^\bla)$ is irreducible and equivalent to $\pi^\bla$.

\end{itemize}
Next, due to (1) and because the action of $\wt\T^\bla(\De(k,\varphi))$ on $H_\ell(T^\bla)$ coincides with that of $\T^\bla_\ell(\De_\ell(k,\varphi))$, the scalar $\wh\De(k,\psi)(\bla)$ is the (only) eigenvalue of the scalar operator $\T^\bla_\ell(\De_\ell(k,\varphi))$. 

Rewrite the definition \eqref{eq10.A} for the function $\varphi$ from \eqref{eq10.C}:
\begin{equation}\label{eq10.D}
\De_\ell(k,\varphi):=\left(\frac{\dim\psi}{|G|}^k\right) \sum \bar\psi(g_k\dots g_1)g_1^{(i_1)}\dots g_k^{(i_k)}(i_1,\dots,i_k)\beps_{i_1}\dots\beps_{i_k},
\end{equation}
where the sum is taken over all ordered $k$-tuples of pairwise distinct indices $i_1,\dots,i_k$ in $\N_\ell$ and aribitrary ordered $k$-tuples of elements $g_1,\dots,g_k$ in $G$. 

Next, we examine separately two possible cases, $k\le \ell$ and $k>\ell$. 

Assume $k\le\ell$. Then, by virtue of (2), the elements $\beps_i=1-\eps_i$ with indices $1\le i\le\ell$ act on the subspace  $H_\ell(T^\bla)$ as the identity operator. Comparing \eqref{eq10.D} with the definition \eqref{eq8.A} of the central elements $c^{(k,\psi)}_n\in\ZZ\C[G(\ell)]$ we see that on the subspace $H_\ell(T^\bla)$ we have
$$
\T^\bla_\ell(\De_\ell(k,\varphi))=T^\bla(c^{(k,\psi)}_\ell).
$$
By virtue of (3), the element $c^{(k,\psi)}_\ell$ acts on this subspace as a scalar operator. Now we apply Proposition \ref{prop8.A} which tells us that the corresponding scalar, $\wh c^{(k,\psi)}_\ell(\la)$, equals $p^\#_{k,\psi}(\bla)$. This proves \eqref{eq10.B} in the case $k\le\ell$. 

Assume now $k>\ell$. By Definition \ref{def10.C} (ii), we have $\De_\ell(k,\varphi)=0$ and hence $\wh\De(k,\varphi)(\bla)=0$. On the other hand, observe that $k>\ell=\Vert\bla\Vert$ implies $p^\#_{k,\psi}(\bla)=0$. Indeed, to see this, we use the definition of $p^\#_{k,\psi}(\bla)$ (see \eqref{eq8.B}) and the fact that $p^\#_\rho(\nu)=0$ for $|\rho|>|\nu|$ (see \eqref{eq3.K}). Thus, in the case $k>\ell$ the equality \eqref{eq10.B} holds true because both its sides equal $0$. 

This concludes the proof. 
\end{proof}

\begin{corollary}[cf. Corollary \ref{cor5.B}]
The correspondence $a\mapsto\widehat a(\cdot)$ determined by \eqref{eq10.E} establishes an isomorphism $\A_0(G)\to \SSym(G)$ of filtered algebras.
\end{corollary}

\begin{proof}
By the very definition of the correspondence in question, it is an injective algebra morphism $\A_0(G)\to \SSym(G)$. By Proposition \ref{prop10.A}, the elements  $\De_\ell(k,\varphi)$ from Proposition \ref{prop10.C} are algebraically independent generators of $\A_0(G)$. Next, Proposition \ref{prop10.C} says that their images are the functions of the form $p^\#_{k,\psi}(\bla)$, where $k=1,2,\dots$ and $\psi\in\Psi(G)$. Finally, these functions generate the algebra isomorphic to $\SSym(G)$. Thus, we obtain an algebra isomorphism. The fact that it preserves the filtration is obvious, because 
$$
\deg \De_\ell(k,\varphi)=\deg p^\#_{k,\psi}=k.
$$
This concludes the proof. 
\end{proof}

\begin{proposition}[cf. Proposition \ref{prop5.A}]\label{prop10.D}
Let $\bla\in \Y(\Psi)$. For any $a\in \A(G)$, the corresponding operator $\wt\T^\bla(a)$, which is initially defined on the dense subspace $H_\infty(T^\bla)$, is bounded and hence can be extended to the whole Hilbert space $H(T^\bla)$.
\end{proposition}

\begin{proof}
The argument is the same as in Proposition \ref{prop5.A}, with the reference to Corollary \ref{cor10.A} instead of Corollary \ref{cor5.A}. 
\end{proof}

\section{Main theorem, case of wreath products $G\wr S(\infty)$}\label{sect11}

\subsection{The embedding $\A(G)\to \prod_{\bla\in\Y(\Psi)}\End(H(T^\bla))$}

By analogy with section \ref{sect6.1} we set 
$$
\mathcal E(G):=\prod_{\bla\in\Y(\Psi)}\End(H(T^\bla));
$$
it is an algebra with respect to component-wise operations. By virtue of Proposition \ref{prop10.D},  the correspondence $a\mapsto \wt\T^\bla(a)$ is an algebra morphism $\A(G)\to\End(T^\bla)$ for each $\bla\in\Y(\Psi)$. The total collection of these morphisms gives rise to an algebra morphism 
\begin{equation}\label{eq11.A}
\iota_{\A(G)}: \A(G)\to  \mathcal E(G), \quad a\mapsto \prod_{\bla\in\Y(\Psi)}\wt\T^\bla(a).
\end{equation}
It is injective: the proof is the same as in Lemma \ref{lemma6.A}.

\subsection{The algebra $\B(G)$ --- the virtual group algebra of $G(\infty)$}

Here we have to reproduce, with suitable minor changes, the material of section \ref{sect6.2}. 
 
For $0\le m\le n$ we denote by  $\B_m(n,G)$ the centralizer of $G_m(n)$ in $\C[G(n)]$. In particular, $\B_0(n,G)$ coincides with $\ZZ\C[G(n)]$, the center of the group algebra $\C[G(n)]$.

\begin{definition}[cf. Definition \ref{def6.A}]

(i) Let  $\al=(\al_1,\al_2,\dots)$ be an infinite sequence of elements of $\C[G(\infty)]$ such that $\al_n\in\C[G(n)]$ for all $n$ large enough. We say that $\al$ is \emph{convergent} if the following three conditions are satisfied:

\begin{itemize}
\item for any $\bla\in\Y(\Psi)$, the operators $T^\bla(\al_n)$  on the space $H(T^\bla)$ have a limit,  $\Tb^\la(\al)$, in the strong operator topology; 
\item there exists $m$ such that $\al_n\in\B_m(n,G)$ for all $n$ large enough;
\item $\sup \deg \al_n<\infty$. 
\end{itemize}

(ii) Next, we say that two convergent sequences, $\al$ and $\be$, are \emph{equivalent} if $\Tb^\bla(\al)=\Tb^\bla(\be)$ for all $\bla\in\Y(\Psi)$. 
\end{definition}

\begin{definition}[cf. Definition \ref{def6.B}]
Let $\B(G)$ be the set of equivalence classes of convergent sequences. We endow it with a natural structure of associative algebra, induced by component-wise operations on convergent sequences. 
\end{definition}

Again, the correctness of the definition is verified using the argument from Lemma \ref{lemma3.B}. 

For each $\bla\in\Y(\Psi)$, the correspondence $\al\mapsto \Tb^\bla(\al)$ determines a representation of algebra $\B(G)$ by bounded operators on the space $H(T^\bla)$. 

There is a natural embedding $\C[G(\infty)]\to\B(G)$ whose image is formed by the sequences $\al=(\al_n)$; here we argue as in Lemma \ref{lemma6.B}. 

We call $\B(G)$ the \emph{virtual group algebra} of $G(\infty)$.

\subsection{The isomorphism $\B(G)\to\A(G)$}

For each $m=0,1,2,\dots$, we denote by $\B_m(G)$ the subalgebra of $\B(G)$ formed by the convergent sequences $(\al_n)$ such that $\al_n\in\B_m(n,G)$ for all $n$ large enough. By the very definition of the algebra $\B(G)$, we have an algebra embedding (cf. \eqref{eq6.A})
\begin{equation}\label{eq11.B}
\iota_{\B(G)}: \B(G)\to  \mathcal E(G), \quad \al\mapsto \prod_{\bla\in\Y(\Psi)}\Tb^\bla(\al).
\end{equation}

\begin{theorem}[cf. Theorem \ref{thm6.A}]\label{thm11.A}
There exists an algebra isomorphism $\F:\B(G)\to\A(G)$, uniquely determined by the property that $\iota_{\A(G)}\circ\F=\iota_{\B(G)}$, meaning that  the diagram 
$$
  \xymatrix{ \B(G) \ar[r]^\F \ar[dr]_{\iota_{\B(G)}}& \A(G) \ar[d]^{\iota_{\A(G)}} \\
  & \mathcal E(G)}
$$
is commutative.

Furthermore, $\F(\B_m(G))=\A_m(G)$ for each $m=0,1,2,\dots$\,. 
\end{theorem}

\begin{proof}
All steps of the proof of Theorem \ref{thm6.A} extend smoothly. The auxiliary results used in Theorem \ref{thm6.A} have exact counterparts, described in sections \ref{sect6} -- \ref{sect10}.  In particular, Theorem \ref{thm8.A} serves as a substitute of Theorem \ref{thm3.A}, and Corollary \ref{cor10.A} is an extension of Corollary \ref{cor5.A}.\end{proof}

\end{document}